\date{May 6, 2015}
\documentclass[reqno]{amsart}
\usepackage{latexsym,amsmath,amsfonts,amscd,amssymb,stmaryrd}
\usepackage{graphics}
\usepackage{color,framed,hyperref,linegoal}

\definecolor{SolutionColor}{rgb}{1,1,0.5}

\theoremstyle{plain}
\newtheorem{theorem}{Theorem}[section]
\newtheorem{corollary}[theorem]{Corollary}
\newtheorem{lemma}[theorem]{Lemma}
\newtheorem{proposition}[theorem]{Proposition}
\theoremstyle{definition}
\newtheorem{definition}[theorem]{Definition}
\newtheorem{remark}[theorem]{Remark}
\newtheorem{exemple}[theorem]{Example}
\renewcommand{\a} \alpha
\renewcommand{\b} \beta

\newcommand{\G}{\Gamma}

\newcommand{\QQ}{\mathbb{Q}}

\newcommand{\CC}{\mathbb{C}}

\newcommand{\ZZ}{\mathbb{Z}}
\newcommand{\vvarpi}{e^{2\pi i /3}}

\newcommand{\co}{\colon\,}

\newcommand{\x}{\times}

\newcommand{\la}{\langle}
\newcommand{\ra}{\rangle}

\newcommand{\Id}{\mathrm{Id}}

\newcommand{\Sym}{\mathrm{Sym}}
\newcommand{\res}{\mathrm{res}}

\DeclareMathOperator{\Hom}{Hom}
 \DeclareMathOperator{\diag}{diag}

\DeclareMathOperator{\GL}{GL}
\DeclareMathOperator{\PGL}{PGL}
\DeclareMathOperator{\PSL}{PSL}
\DeclareMathOperator{\SL}{SL}

\DeclareMathOperator{\tr}{tr}

\title{The $\SL(3,\CC)$-character variety of the figure eight knot}


\author[M. Heusener]{Michael Heusener}
\address{Laboratoire de Math\'ematiques, UMR 6620 du CNRS, 
Universit\'e Blaise Pascal - Clermont II, F-63177 Aubi\`ere Cedex, France}
\email{Michael.Heusener@math.univ-bpclermont.fr}

\author[V. Mu\~{n}oz]{Vicente Mu\~{n}oz}
\address{Facultad de Matem\'aticas, Universidad Complutense de Madrid,
Plaza Ciencias 3, 28040 Madrid, Spain}
\email{vicente.munoz@mat.ucm.es}

\author[J. Porti]{Joan Porti}
\address{Departament de Matem\`atiques. Universitat Aut\`onoma de Barcelona, 08193 Bellaterra. Spain}
\email{porti@mat.uab.cat}

\begin{document}

\begin{abstract}
 We give explicit equations that describe the character variety of the figure eight knot 
 for the groups $\SL(3,\CC)$, $\GL(3,\CC)$ and $\PGL(3,\CC)$. This has five components of
 dimension $2$, one consisting of totally reducible representations, another one
 consisting of partially reducible representations, and three components of irreducible
 representations. Of these, one is distinguished as it contains the curve of
 irreducible representations coming from  $\Sym^2:\SL(2,\CC) \to \SL(3,\CC)$.
 The other two components are induced by exceptional Dehn fillings of the
 figure eight knot. We also describe the action of the symmetry group
 of the figure eight knot on the character  varieties.
\end{abstract}

\maketitle

\section{Introduction} \label{sec:intro}

Since the foundational work of Thurston  \cite{ThurstonNotes, ThurstonBullAMS}  and Culler and Shalen \cite{CS}, 
the varieties of representations and characters of three-manifold groups in $\SL(2,\CC)$
have been intensively studied, as they reflect geometric and topological properties of the three-manifold.
In particular they have been  used to study knots $K\subset S^3$, by analysing the
$\SL(2,\CC)$-character variety of the fundamental group of the knot complement
$S^3-K$ (these are called \emph{knot groups}). 

Much less is known for the character varieties of three-manifold groups in other Lie groups, notably
for $\SL(r,\CC)$ with $r\geq 3$. There has been an increasing interest for those
in the last years. For instance, inspired in the A-coordinates in higher 
Teichm\"uller theory of Fock and Goncharov \cite{FG06}, some authors have used
the so called Ptolemy coordinates for studying spaces of representations,
based on subdivisions of ideal triangulations of the three-manifold. Among others, 
we mention the work of  Dimofty, Gabella, Garoufalidis, Goerner, Goncharov, Thurston, 
and Zickert \cite{DGG,GTZ, GaroufalidisZickert, GaroufalidisGoenerZickert, DimoftyGaroufalidis}.
Geometric aspects of these representations, including volume and rigidity, have been addressed by 
Bucher, Burger, and Iozzi in \cite{BBI},
and by Bergeron, Falbel, and Guilloux in \cite{BFG}, who view these representations as holonomies of marked  flag
structures. We also recall the work 
Deraux and Deraux-Falbel in \cite{D1, D2, DF} to study CR and complex hyperbolic structures.
 
In a recent preprint,  
Falbel,  Guilloux,   Koseleff,  Rouillier,   and Thistlethwaite \cite{Falbel} compute 
the variety of characters of  the figure eight knot in $\SL(3,\CC)$  using the ideal triangulation approach.
We also compute  this variety of characters in this  paper, but with a completely different method and 
we obtain a different description. Here we describe it as an affine algebraic set with trace functions as 
coordinates.

The $\SL(3,\CC)$-character varieties for
free groups have also been studied in \cite{Lawton0,Lawton1,LawtonP,Will}, and 
the $\SL(3,\CC)$-character variety of torus knot groups has been determined in \cite{MunPor}.
Other local results have been proved in \cite{BFGKR,HM14,BenAH15,MenalP}.

% Moreover, it follows from Theorem~0.4 in \cite{MenalP} that the $\SL(r,\CC)$-character variety of
% a hyperbolic knot has at least one $(r-1)$-dimensional component.
% We will call such an algebraic component a \emph{distinguished component} 
% (see Remark~\ref{rem:distinguished}). Other existence results for irreducible representations were proved in \cite{HM14,BenAH15}.
%  

% 
% 
% 
% \comment{M: Perhaps we can put here something like the following paragraph ? It would put our work into a bigger frame\ldots
% 
% Recently, motivated by the work of Fock and Goncharov \cite{FG06} on higher Teichmüller spaces,
% some authors studied special coordinates on representation varieties coming from triangulations
% (see \cite{BFG,DGG,GGTZ}).
% Their goal is to understand structural properties of these coordinates, to decide which representations give rise to geometric structures, and to provide explicit calculations.
% 
% 
% In the  articles \cite{MenalP,BF3,HeusenerP,HM14} the authors are pursuing a different approach to investigate representation varieties: starting with a representation of the fundamental group of a $3$-manifold which has a geometric meaning, they study the deformations of this representation.
% In same cases it is possible to describe the local structure of the representation and character varieties.}

For $\G$ a finitely generated group, and for $G=\SL(r,\CC)$, $\GL(r,\CC)$, or 
$\PGL(r,\CC)$, the \emph{variety of representations}  is denoted by $ R(\Gamma, 
G)$. It is an algebraic affine set, the action of $G$ by conjugation is 
algebraic and the affine GIT quotient is naturally identified with the 
\emph{variety of characters} $X(\Gamma, G)$ \cite{LM}. Notice that both $ 
R(\Gamma, G)$ and $X(\Gamma, G)$ can be reducible (hence not varieties in the 
usual sense), and their defining polynomial ideals may be non-radical (when this 
happens they are said to be scheme non-reduced), cf.~\cite{LM}. When 
$G=\SL(r,\CC)$, points in $X(\Gamma, \SL(r,\CC))$ are precisely characters of 
representations, i.e. for $\rho\in R(\Gamma, G)$ its character is the map $ 
\chi_\rho\co \Gamma\to \CC$ defined by $\chi_\rho(\gamma)=\tr(\rho(\gamma))$, 
$\forall\gamma\in\Gamma$  \cite{LM}.

% Let us introduce some notation.
% Let $\G$ be a finitely presented group, and let $G=\SL(r,\CC)$, $\GL(r,\CC)$ or $\PGL(r,\CC)$. 
% A \textit{representation} of $\G$ in $G$ is a homomorphism $\rho: \G\to G$.
% Consider a presentation $\G=\la x_1,\ldots, x_k | r_1,\ldots, r_s \ra$. Then $\rho$ is completely
% determined by the $k$-tuple $(A_1,\ldots, A_k)=(\rho(x_1),\ldots, \rho(x_k))$
% subject to the relations $r_j(A_1,\ldots, A_k)=\id$, $1\leq j \leq s$. The space
% of representations is
%  \begin{eqnarray*}
%  R(\G,G) &=& \Hom(\G, G) \\
%   &=& \{(A_1,\ldots, A_k) \in G^k \, | \,
%  r_j(A_1,\ldots, A_k)=\id, \,  1\leq j \leq s \}\subset G^{k}\, .
%  \end{eqnarray*}
% Therefore $R(\G,G)$ is an affine algebraic set.
% We say that two representations $\rho$ and $\rho'$ are
% equivalent if there exists $P\in G$ such that $\rho'(g)=P^{-1} \rho(g) P$,
% for every $g\in G$. 
% This produces an action of $\PGL(r,\CC)$ in $R(\G,G)$. The moduli space of representations
% is the GIT quotient
%  $$
%  M(\G,G) = R(\G,G) \sslash G \, .
%  $$
 \begin{definition}\label{def:reducible}
A representation $\rho$ is \textit{reducible} if there exists some proper 
subspace $V\subset \CC^r$ such that for  all $g\in G$ we have 
$\rho(g)(V)\subset V$;  otherwise $\rho$ is
\textit{irreducible}. A \emph{semisimple} representation is  a direct sum of 
irreducible representations. 

A representation $\rho\co\Gamma\to\SL(3,\CC)$ is called \emph{partially reducible} 
(respectively \emph{totally reducible}) if it is the sum of a one-dimensional and two-dimensional irreducible representation (respectively a sum of three 
one-dimensional representations).
\end{definition}
% 
% The space $M(\G,G)$ parametrizes orbits of semisimple representations
% \cite[Thm.~ 1.28]{LM}. Suppose now that $G=\SL(r,\CC)$. 
% %Here a representation is called \emph{semisimple} if it is the sum of irreducible representations.
% Given a representation $\rho: \G\to G$, we define its
% \textit{character} as the map $\chi_\rho: \G\to \CC$,
% $\chi_\rho(g)=\tr \rho (g)$. Note that two equivalent
% representations $\rho$ and $\rho'$ have the same character.
% There is a character map $\chi: R(\G,G)\to \CC^\G$, $\rho\mapsto
% \chi_\rho$, whose image
%  $$
%  X(\G,G)=\chi(R(\G,G))
%  $$
% is called the \textit{character variety of $\G$}. 
% The natural algebraic map $M(\G,G)\to X(\G,G)$ 
% is an  isomorphism (see Chapter 1 in \cite{LM}). 
% As the space of representations is Noetherian, 
% there exists a finite collection $g_1,\ldots, g_t \in G$ 
% such that $\chi_\rho$ is determined by $\chi_\rho(g_1),\ldots,
% \chi_\rho(g_t)$, for any $\rho$. Such collection gives an isomorphism
%  $$
%  X(\G,G) \cong \mathbf{V}(I) \subset \CC^t,
%   $$
% where the coordinates are $x_i(\rho)=\chi_\rho(g_i)$ and $I$ is the ideal
% of relations satisfied by the $x_i$ on $X(\G,G)$. 
% \begin{remark}
% Notice that the algebraic equations defining the representation
% variety may be non-reduced, hence there are underlying affine schemes 
% $\mathcal{R}( \Gamma, \mathrm{SL}_n( \mathbf C))$ and
% $\mathcal{X}( \Gamma, \mathrm{SL}_n( \mathbf C))$ with  possible non-reduced
% coordinate rings.
% \end{remark}

This paper focuses on the fundamental group of the complement of
the figure eight knot, denoted by  $\Gamma$. 
We consider the following two natural presentations:
\begin{align}
 \Gamma & \cong \la a,b,t \mid tat^{-1} =ab , t b t^{-1}=bab \ra\,, \label{eq:pres-fiber}\\
 & \cong \la S,T \mid ST^{-1}S^{-1}T\,S = T\, ST^{-1}S^{-1}T \ra\,. \label{eq:pres-2bridge}
\end{align}
The former comes from the fibration of the three-manifold over the circle, the latter   from knot theory \cite{Rolfsen, BZH2013},
as the generators are meridian loops.
The presentations are related by
\[
\begin{cases}
S&= t, \\
T&= a^{-1} t  a,
\end{cases} 
\quad\text{ and }\quad
\begin{cases}
t&= S ,\\
a&= T^{-1}STS^{-1},\\
b&= T S^{-1}.
\end{cases} 
 \]
The figure eight knot exterior fibres over the circle, with fibre a punctured torus, this explains 
Presentation~\eqref{eq:pres-fiber}, as the group of the fiber is freely generated by $a$ and $b$.
This free group
$F_2=\langle a, b\rangle$  is the kernel of the abelianisation 
$\varphi\co\Gamma\to\Gamma_\mathrm{ab}\cong\ZZ$, which is given by 
 \begin{equation}\label{eqn:varphi}
  \varphi(S)=\varphi(T)=1, \quad \text{ and } \quad
\varphi(t)=1,\  \varphi(a)=\varphi(b)=0,
 \end{equation}
respectively; so $F_2$ is also the commutator subgroup of $\Gamma$.
Hence, for any representation $\rho\co\Gamma \to G$
with $G= \SL(r,\CC)$, $\GL(r,\CC)$, or $\PGL(r,\CC)$,  we have
\begin{equation}
 \label{eqn:rhoF2inSL}
\rho(F_2)=\langle\rho(a),\rho(b) \rangle\subset \SL(r,\CC)\, . 
 \end{equation}
 
Presentation~\eqref{eq:pres-2bridge} is the usual presentation for a 2-bridge knot group, 
as the figure eight knot is the $2$-bridge knot $\mathfrak{b}(5,3)$ in Schubert's notation (see \cite[12.A]{BZH2013}).
%because $F_2$ is the commutator subgroup $[\Gamma,\Gamma]$.
%More precisely, $F_2$ is the kernel of the Abelianization 
%$\Gamma\to\Gamma_\mathrm{ab}\cong\ZZ$ which is given by $S\mapsto 1$ and $T\mapsto 1$.
%
%Let  $\mu_r = \langle \varpi\rangle \subset \CC^* $  denote the group of $r$-th roots of unit.
%For every knot group $\Gamma$ we have
%\begin{align*}
% X(\Gamma,\GL(r,\CC)) &=
% (X(\Gamma,\SL(r,\CC)) \x \CC^*)/\mu_r \, ,\\
% & = X(\Gamma,\SL(r,\CC)) \x_{\mu_r} \CC^* \,.
%\end{align*}
%\colorbox{yellow}{Add proof}\\
%%
%%where $(A,B,T,\lambda) \in X(\Gamma,\SL(r,\CC)) \x \CC^*$ is sent to 
%%$(A,B,\lambda T) \in X(\Gamma,\GL(r,\CC))$, and $\varpi\in\mu_r$ acts 
%%as $(A,B,T,\lambda) \mapsto (A,B, \varpi T, \varpi^{-1}\lambda)$.
%%Here $A$, $B$, and $T\in G$ denote the images of $a$, $b$, and $t$ respectively, 
%%and $\mu_r\subset \CC^* $ the group of $r$-th roots of unit.
%%
%Moreover, for a knot group $\Gamma$ every representation 
%$\Gamma\to\PGL(r,\CC) \cong \PSL(r,\CC)$ lifts to a representation to
%$\Gamma\to\SL(r,\CC)$. This follows from $H^2(\Gamma; \mu_r)=0$.
%Moreover, the lifts are parametrized by $H^1(\Gamma; \mu_r)\cong \mu_r$.
%Thus
% $$
% X(\Gamma,\PSL(r,\CC)) =  X(\Gamma,\SL(r,\CC)) /\mu_r,
% X(\Gamma,\GL(r,\CC))/ \CC^*  \, ,
% $$ 
%where $\mu_r$ and $\CC^*$ are acting by multiplication.
%
%By the above, we have for every fibered knot group and in all cases  $G= \SL(r,\CC)$,  $\GL(r,\CC)$ or $\PGL(r,\CC)$, a restriction map 
% $$
% X(\Gamma, G) \to X([\Gamma,\Gamma], \SL(r,\CC)).
% $$
% Recall that the commutator group is finitely generated if and only if the knot is fibered.
In particular $X(\Gamma,\SL(3,\CC))$ is a subvariety of the 
$\SL(3,\CC)$-character variety of the free group of rank $2$
generated by $S$ and $T$, and
results of \cite{Lawton0,Will} will apply, see Proposition~\ref{prop:parametersLawton}.
More precisely, we consider the algebraic map
 $X(\Gamma,\SL(3,\CC))\to \CC^8$ defined by
 $\chi\mapsto \big(y(\chi), \bar y(\chi), z(\chi), \bar z(\chi), \alpha(\chi),\bar\alpha(\chi),
 \beta(\chi),\bar\beta(\chi)\big)$, where
 \begin{alignat}{4}
 y(\chi) &= \chi(S),\ & \bar{y}(\chi) &=\chi(S^{-1}),\ & z(\chi)&=\chi(ST),\ &\bar{z}(\chi)&=\chi(S^{-1}T^{-1}), \notag\\
 \alpha(\chi) &= \chi([T,S]),\ & \bar\alpha(\chi) &= \chi([S,T]),\ &
\beta(\chi) &= \chi(S^{-1} T),\ &\bar{\beta}(\chi) &=  \chi(S T^{-1}).  \label{eq:parameters}
\end{alignat} 
We see in Proposition~\ref{prop:parametersLawton}
 that they define coordinates  for $X(\Gamma,\SL(3,\CC))$ as subvariety of  $\CC^8$.
Using Presentation~\eqref{eq:pres-fiber}, those coordinates are:
\begin{alignat}{4}
\alpha(\chi)&=\chi(a),\quad &   \beta(\chi)&=\chi(b), \quad &  y(\chi) &=\chi(t), \quad & z(\chi)&=\chi(t a^{-1}t a)=\chi(t^2b), \notag\\
\bar \alpha(\chi)&=\chi(a^{-1}),\quad &   \bar \beta(\chi)&=\chi(b^{-1}), \quad & 
 \bar y(\chi) &=\chi(t^{-1}), \quad & \bar z(\chi)&=\chi(a^{-1} t^{-1} a t^{-1}  )= \chi(b^{-1} t^{-2}). \label{eq:parametersfib}
\end{alignat}

%
% 
% 
% \begin{proposition}\label{prop:pre-parameters}
% Let $\Gamma$ be a 2-bridge knot group. Then the character variety 
% $X(\Gamma,\SL(3,\CC))$ can be embedded in $\CC^9$. More precisely, the embedding is given by
% \[
% \chi\mapsto \big(\chi(S), \chi(S^{-1}), \chi(ST), \chi(S^{-1}T^{-1}), \chi(S T^{-1}), \chi(S^{-1} T)
% \chi([S,T]), chi([T,S])\big)\,. 
% \]
%\end{proposition}

Throughout this paper $\mu_3=\{1,\varpi,\varpi^2\}\subset\CC^*$,
$\varpi=e^{2\pi i/3}$, denotes  the group of the third roots of unity. 
We  identify $\mu_3$ with the center of $\SL(3,\CC)$, consisting of diagonal matrices, and for any knot group 
$\Gamma$, it acts on 
$R(\Gamma,\SL(3,\CC))$ and $X(\Gamma,\SL(3,\CC))$ via
 \[ 
 (\varpi\rho) (\gamma) = \varpi^{\varphi(\gamma)}\rho(\gamma) \quad 
 \text{ and } \quad
 (\varpi\chi) (\gamma) = \varpi^{\varphi(\gamma)} \chi(\gamma),
 \]
respectively,  where $\varphi\co\Gamma\to\ZZ$ is the abelianization map in (\ref{eqn:varphi}).
% The action of the generator of the center $\varpi\in\mu_3$ is given by
% \[ \varpi\cdot  (y,\bar y,z,\bar z,\alpha,\bar\alpha,\beta,\bar\beta)=
% (\varpi y,\varpi^2\bar y,\varpi^2 z,\varpi\bar z,\alpha,\bar\alpha,\beta,\bar\beta)\,.\]
The action of the generator of the center $\varpi\in\mu_3$ in coordinates is
\[ \varpi\cdot  (y,\bar y,z,\bar z,\alpha,\bar\alpha,\beta,\bar\beta)=
(\varpi y,\varpi^2\bar y,\varpi^2 z,\varpi\bar z,\alpha,\bar\alpha,\beta,\bar\beta)\,.\]
 
The main result of this paper is the following.
\begin{theorem}
\label{thm:main} Let $\Gamma$ be the group of the figure eight knot.
 The character variety $X(\Gamma,\SL(3,\CC))\subset\CC^8$ has five algebraic components.
 They are described in terms of the coordinates \eqref{eq:parameters}
% \begin{alignat*}{4}
% y &= \chi(S),\ & \bar{y} &=\chi(S^{-1}),\ & z&=\chi(ST),\ &\bar{z}&=\chi(S^{-1}T^{-1}), \\
%\beta &= \chi(S^{-1} T),\ &\bar{\beta} &=  \chi(S T^{-1}),\ & \alpha &= \chi([T,S]),\ &
%\bar\alpha &= \chi([S,T])
%\end{alignat*}
 as follows:
  \begin{enumerate}
  \item The component $X_\mathrm{TR}$  corresponding to totally reducible representations is described by:
  \[ 
  \alpha=\bar\alpha=\beta=\bar\beta=3,\quad (y,\bar{y})\in \CC^2,\quad  
  z=y^2-2\bar{y}, \quad \bar{z}=\bar{y}^2-2y\,.
  \]
  The component $X_\mathrm{TR}$ is smooth and isomorphic to $\CC^2$.
  %\colorbox{yellow}{ I have changed $z$ and $\bar z$}
  \item The component $X_\mathrm{PR}$ corresponding to partially reducible representations is parametrized by the smooth variety 
  \[ \mathcal{P} =\{ (v,w,x_1)\in \CC\x\CC^*\x(\CC-\{1\})\mid \frac{x_1^2 +x_1-1}{x_1-1} w = v^2\}\,.\]
More precisely, a parametrization 
$\Phi\co\mathcal{P}\to X_\mathrm{PR}$ is given by:
  \begin{alignat*}{2}
    \alpha(v,w,x_1)=\bar\alpha(v,w,x_1)&= x_1+1, \quad & 
    \beta(v,w,x_1)=\bar\beta(v,w,x_1)&=\frac{x_1}{x_1-1} + 1 ,\\ 
  y(v,w,x_1)&=v+\frac1w,\quad &  \bar{y}(v,w,x_1)&=w+\frac{v}{w}, \\
z(v,w,x_1)&=w\, \alpha+\frac{1}{w^2},\quad &
\bar z(v,w,x_1)&=\frac{\alpha}{w}+w^2\,.
  \end{alignat*}  
The component $X_\mathrm{PR}$ is smooth except at the three points
\[
(y,\bar y,z,\bar z,\alpha,\bar\alpha,\beta,\bar\beta)\in \mu_3 \cdot(4,4,8,8,3,3,3,3)\,.
\]
%where there are nodes (two smooth branches intersecting transversally at the point).
  \end{enumerate}
  There are three components $V_0$, $V_1$ and $V_2$
corresponding to irreducible representations. 
\begin{enumerate}\setcounter{enumi}{2}
\item The distinguished component $V_0$ 
is the zero set of the ideal generated by the following equations:
  \begin{align*}
   \alpha &=\bar\alpha ,\quad \beta=\bar\beta , \\ 
   %\eta^2 - P \eta + Q =0, \\
%    &  yz = (\alpha+1)(\beta+1), \\
%    & y^3 + z^3 = \alpha^2\beta+\alpha\beta^2+ 6 \alpha\beta+3 \alpha+ 3 \beta+2, \\
%    & (\alpha-\beta) (2\eta-(\alpha^2\beta^2-2\alpha^2\beta-2\beta^2\alpha+2\alpha^2+2\beta^2-3))= \\
%     & \qquad \qquad =(\alpha\beta-2\alpha-2\beta+3 )
% (2y^3-(\alpha^2\beta+\alpha\beta^2+ 6 \alpha\beta+3 \alpha+ 3 \beta+2)), \\
%  &***********************
%   \end{align*}
%   
%     \begin{align*}
     y  \bar{y} &= (\alpha+1)(\beta+1), \\
       z\bar z&= 2\alpha^2\beta+\alpha^2+1, \\
    y^3 + \bar{y}^3 &= \alpha^2\beta+\alpha\beta^2+ 6 \alpha\beta+3 \alpha+ 3 \beta+2, \\
 z^3 +\bar z^3 &= \alpha^4 \beta^2  + 10 \alpha^2\beta+ 9\alpha^2 - 2 \alpha^3 -2  ,\\
 yz +\bar y \bar z &= \alpha^2\beta+3\alpha\beta+ 3\alpha+1, \\
  \bar{y}^{2} z + y^{2} \bar{z} &=  \alpha^{2} \beta^{2}  + 4
\alpha^{2} \beta + 2 \alpha^{2} + 4 \alpha \beta + 2 \alpha + 2 \beta + 1,\\
 \bar{y} z^{2} + y\bar{z}^{2} &= \alpha^{3} \beta^{2} +  \alpha^{3} \beta +  4 \alpha^{2} \beta 
+ 3 \alpha^{2} + 5 \alpha \beta + 3\alpha - 1\,.
%   & (\alpha-\beta) (2\eta-(\alpha^2\beta^2-2\alpha^2\beta-2\beta^2\alpha+2\alpha^2+2\beta^2-3))= \\
%    & \qquad \qquad \qquad \qquad =(\alpha\beta-2\alpha-2\beta+3 )
%(2y^3-(\alpha^2\beta+\alpha\beta^2+ 6 \alpha\beta+3 \alpha+ 3 \beta+2)), \\
% &(1-\alpha)(2\eta-(\alpha^2\beta^2-2\alpha^2\beta-2\beta^2\alpha+2\alpha^2+2\beta^2-3))= \\
% &\qquad\qquad \qquad \qquad  = (\alpha\beta-2\alpha-2\beta+3 ) (2\bar{y}\bar z 
% - (\alpha^2\beta+3\alpha\beta+ 3\alpha+1)), \\
% &
% (\alpha^3\beta+3 \alpha^2-4 \alpha)(2\eta-
% (\alpha^2\beta^2-2\alpha^2\beta-2\beta^2\alpha+2\alpha^2+2\beta^2-3)) = \\
%&\qquad \qquad \qquad \qquad = (\alpha\beta-2\alpha-2\beta+3 ) (2z^3 -
%(\alpha^4 \beta^2  + 10 \alpha^2\beta+ 9\alpha^2 - 2 \alpha^3-2)) ,
  \end{align*} 
% where $P,Q$ are given in Lemma \ref{lem:Wj}. This is called distinguished because
% it contains the image of the irreducible characters via $\Sym^2: \SL(2,\CC)\to \SL(3,\CC)$.
\item   The first non distinguished component  $V_1$: 
\begin{alignat*}{2}
    \alpha&=\bar\alpha=1, &
     y\bar{y} &= \beta+\bar \beta+2, \\
    y^3 + \bar{y}^3 &= \beta\bar\beta+ 5 \beta+ 5 \bar\beta+ 5,\quad &
    \bar{z}&=y,\quad z=\bar{y}\,.
   \end{alignat*}
\item   The second non distinguished component $V_2$: 
\begin{alignat*}{2}
    \beta&=\bar\beta=1 ,  &
     y\bar{y} &= \alpha+\bar \alpha+2, \\
    y^3 + \bar{y}^3 &=  \alpha\bar\alpha + 5 \alpha+ 5 \bar\alpha+ 5, \quad &
%   & \eta  =\bar{y}^3 -3(\alpha+\bar\alpha+1), \\
    z &= y^2 - \bar{y},\quad \bar{z} = \bar{y}^2 -y\,.
   \end{alignat*}
\end{enumerate}   
The components $V_i$, $i=0,1,2$,  are smooth except at the three points:
$\mu_3 \cdot(2,2,2,2,1,1,1,1)$.  
\end{theorem}
 
The component $V_0$ is called distinguished because it contains the composition of the holonomy
representation of the complete hyperbolic structure with the irreducible representation 
$\Sym\co \PSL(2,\CC)\to \SL(3,\CC)$. The components $V_1$ and $V_2$ factor through Dehn fillings of the knot, 
with respective slopes $\pm 3$, see Proposition~\ref{prop:non-dist-components}.
In particular they do not contain faithful representations.
 
 \begin{remark}
The ideal generated by the equations in item~(3) of Theorem~\ref{thm:main} is not radical. 
 Generators of the radical are given in Remark~\ref{rem:radical}, those are the defining polynomials
 of the variety of characters, as we know that $V_0$ is scheme reduced, by Proposition~\ref{prop:smooth}.
 \end{remark}
 
The intersections of the components are as follows:
 \begin{itemize}
  \item $X_\mathrm{TR}\cap X_\mathrm{PR}$ is the curve $\alpha=\bar\alpha=\beta=
\bar\beta=3$, $y^2 \bar{y}^2 - 5 y^3 - 5 \bar{y}^3 + 28 y \bar{y} - 64$, $z = y^2-2\bar{y}$, and 
$\bar{z}=\bar{y}^2-2y$.
This curve is smooth except at the three points $\mu_3\cdot(4,4,8,8,3,3,3,3)$.
\item $X_\mathrm{TR} \cap V_1 = X_\mathrm{TR} \cap V_2 = \emptyset$.
\item $X_\mathrm{TR}\cap V_0 = \mu_3\cdot(4,4,8,8,3,3,3,3)$.
  \item The intersections $X_\mathrm{PR}\cap V_1=X_\mathrm{PR}\cap V_2=
  V_0\cap V_1\cap V_2$ consists of three points $\mu_3\cdot(2,2,2,2,1,1,1,1)$.
  These three points are singular points of $V_i$, $i=0,1,2$, and they are regular points on 
 $X_\mathrm{PR}$.
  \item $ X_\mathrm{PR}\cap V_0$ is the curve given by the equations (2) of Theorem~\ref{thm:main}
  and the equation
%   \[  v = \frac{1}{2} \left(w^2 + \frac{1}{w}\right) \ \Longleftrightarrow\ w^3 -2vw +1 =0\,.\]
  \(   w^3 -2vw +1 =0\).
  This curve is non singular except at the three points
\[V_0\cap X_\mathrm{PR}\cap X_\mathrm{TR} = \mu_3\cdot (4,4,8,8,3,3,3,3)\,.\]
  \item $V_0\cap V_1$ is the curve $\alpha=\bar\alpha=1$, $\beta=\bar \beta$, $y \bar{y}= 2\beta+2$,
  $y^3+\bar{y}^3=\beta^2+10\beta+5$. This curve is nonsingular except at 
  $\mu_3\cdot (2,2,2,2,1,1,1,1)$.
  \item $V_0\cap V_2$ is the curve $\alpha=\bar\alpha$, $\beta=\bar \beta=1$, $y \bar{y}= 2\alpha+2$,
  $y^3+\bar{y}^3=\alpha^2+10\alpha+5$. This curve is nonsingular except at 
  $\mu_3\cdot (2,2,2,2,1,1,1,1)$.
\end{itemize}

%The component $V_0$ is smooth except at six points: three of them are again   $V_1\cap V_2=V_0\cap V_1\cap V_2$; the three others
%% 
%% given by  $\alpha=\beta=1$, $\eta=-1$ and  $(y,\bar y)=(\bar z, z)=(2,2)$, $(2\varpi,2\varpi^2)$, or 
%%   $(2\varpi^2,2\varpi)$, and three other by 
%are given by $\alpha=\beta=\eta=3$ and $(y,\bar y,z,\bar z)= (4,4,8,8)$, $(4\varpi,4\varpi^2,8 \varpi^2,8\varpi)$, or 
%  $(4\varpi^2,4\varpi,8\varpi,8\varpi^2)$.
 
%The existence of a two-dimensional distinguished component was proven in \cite{MenalP}. The non-distinguished components occur as representations
%of two exceptional Dehn fillings on the figure eight knot exterior, as it is explained in detail in Section~\ref{section:dehn}.

To obtain the irreducible components, we consider first the restriction of those characters to the group of the fiber
$F_2=\langle a, b\rangle$ (Presentation~\eqref{eq:pres-fiber}) by considering the characters that are fixed by the action of the monodromy. Here we use 
Lawton's coordinates for $X(F_2,\SL(3,\CC))$. This allows to distinguish three components of irreducible characters,
that are worked out explicitly.

The paper is organized as follows. Section~\ref{sec:general} is devoted to 
generalities on character varieties of knot groups.
Representations of $\Gamma$ in $\SL(2,\CC)$, $\GL(2,\CC)$ and $\PGL(2,\CC)$ are 
discussed in 
Section~\ref{section:SL2}, and reducible representations in $\SL(3,\CC)$, in 
Section~\ref{sec:reducible}. Section~\ref{sec:Lawton}
is devoted to the description of the restriction to the variety of characters of 
$F_2$ as fixed points of the monodromy.
Then the non-distinguished and the distinguished components are computed 
respectively in Sections~\ref{sec:nondist} and \ref{sec:distinguished}.
Section~\ref{sec:irredPGL3GL3} is devoted to characters in $\GL(3,\CC)$ and 
$\PGL(3,\CC)$. In Section~\ref{sec:symmetries} we describe
how the symmetry group of the figure eight knot acts on the variety of 
characters. In Section~\ref{section:dehn}
we identify the non-distinguished components as induced by Dehn fillings on the 
knot. Finally, in Section~\ref{sec:explicite}
we discuss explicit representations that are relevant.

Some of the computations require software, either Sage \cite{sage} or  Mathematica 
\cite{mathematica}. All worksheets and notebooks can be found in 
\cite{fig8html}.

\subsection*{Acknowledgements} M.~Heusener and J.~Porti are partially supported by MINECO 
grant MTM2012-34834. V.~Mu\~noz is partially supported by MINECO grant MTM2012-30719.
  
\section{Character varieties of knot groups}
\label{sec:general}

Throughout this section we let $\Gamma$ denote any knot group (in the rest of the paper
it denotes the figure eigth knot exerior), and $\varphi\co\Gamma\to\ZZ$ 
the abelianization which maps the meridian of the knot to $1$ \eqref{eqn:varphi}.
The center $\mu_r$ of $\SL(r,\CC)$ consists of diagonal matrices and it
can be identified with the set of $r$-th roots of unity $\{\varpi^k\mid k=0,\ldots,r-1\}\subset \CC^*$. 
The center acts on 
$R(\Gamma,\SL(r,\CC))$ and $X(\Gamma,\SL(r,\CC))$ via
multiplication, i.e.\ for $\rho\in R(\Gamma,\SL(r,\CC))$,  $\chi\in X(\Gamma,\SL(r,\CC))$,
and $\varpi^k\in\mu_r$,  we have for all $\gamma\in\Gamma$:
\[
\varpi^k\cdot\rho(\gamma) = \varpi^{k\varphi(\gamma)}\rho(\gamma)
\qquad \text{ and }\qquad 
\varpi^k\cdot\chi(\gamma) = \varpi^{k\varphi(\gamma)}\rho(\gamma)\,.
\]

\begin{lemma}\label{lem:lift}
Let $\Gamma$ be a knot group and $\rho\co\Gamma\to\PSL(r,\CC)$ be a representation.
Then there exists a lift $\tilde\rho\co\Gamma\to\SL(r,\CC)$ of $\rho$. Moreover, all lifts of $\rho$ are 
given by $\mu_r\cdot\tilde\rho$.
\end{lemma}
\begin{proof}
There is a short exact sequence
\[ 1\to\mu_r\to\SL(r,\CC)\to\PSL(r,\CC)\to1\,.\]
We associate to the representation $\rho\co\Gamma\to\PSL(r,\CC)$  a second 
Stiefel-Whitney class $w_2 =w_2(\rho) \in H^2(\Gamma,\mu_r)$ defined as follows: choose any set-theoretic lift $ f\co\Gamma\to\SL(r,\CC)$ and define $w_2\co\Gamma\times\Gamma\to\mu_r$ such that 
\[ \forall \gamma_1,\gamma_2\in\Gamma\qquad
 f(\gamma_1\gamma_2)= f(\gamma_1)  f(\gamma_2) w_2(\gamma_1,\gamma_2)\,.\]
It is easy to see that $w_2\in Z^2(\Gamma,\mu_r)$ is a cocyle. Moreover, the cohomology class represented by $w_2$ does not depend on the lift $f$. Now, $w_2$ represents the trivial cohomology class since $H^2(\Gamma,\mu_r)$ is trivial. Therefore, there exists a map $d\co\Gamma\to\mu_r$ such that for all $\gamma_1,\gamma_2\in\Gamma$,
  \[w_2(\gamma_1,\gamma_2)= d(\gamma_1)d(\gamma_2)d(\gamma_1\gamma_2)^{-1}\,. \]
It is clear that
$\tilde\rho\co\Gamma\to\SL(r,\CC)$ given by $\tilde\rho(\gamma)=d(\gamma)f(\gamma)$ is a representation. Finally, $d$ is unique up multiplication with a cocycle 
$h\in H^1(\Gamma,\mu_r)\cong\mathrm{Hom}(\Gamma,\mu_r)\cong\mu_r$. 
\end{proof}

\begin{lemma}\label{lem:GL}
Let $\Gamma$ be a knot group. Then $X(\Gamma,\PSL(r,\CC))\cong X(\Gamma,\SL(r,\CC))/\mu_r$ and  
\[ X(\Gamma,\GL(r,\CC))\cong  X(\Gamma,\SL(r,\CC))\times_{\mu_r} \CC^*.\]
\end{lemma}
\begin{proof}
The isomorphism $X(\Gamma,\PSL(r,\CC))\cong X(\Gamma,\SL(r,\CC))/\mu_r$ follows from Lemma~\ref{lem:lift}. 

Now the same proof as for Lemma~\ref{lem:lift} shows that for any homomorphism 
$h\co\Gamma\to\CC^*$ there exists a homomorphism $\tilde h\co\Gamma\to\CC^*$ such that
$\tilde h^r = h$ (see \cite[Lemma~2.1]{AHJ2010}). Therefore the map
$R(\Gamma,\SL(r,\CC))\times\CC^*\to R(\Gamma,\GL(r,\CC))$
given by $(\rho,\lambda) \mapsto \lambda^\varphi \rho$ is surjective, and
$(\rho,\lambda)$ and $(\rho',\lambda')$ map to the same representation if and only if
$(\rho',\lambda')\in\mu_r(\rho,\lambda)$. Hence,
\[(R(\Gamma,\SL(r,\CC))\times\CC^*)/\mu_r\cong R(\Gamma,\GL(r,\CC))\,.\]
The actions of $\mu_r$ and 
$\GL(r,\CC)$ on $R(\Gamma,\SL(r,\CC))\times\CC^*$ commute. Moreover,
$\GL(r,\CC)$ acts trivially on the representations into the center $\CC^*$. Hence
\begin{align*}
R(\Gamma,\GL(r,\CC))\sslash\GL(r,\CC) &\cong 
\big(R(\Gamma,\SL(r,\CC))\times\CC^*)\sslash \GL(r,\CC)\big)/\mu_r\\
&\cong
(X(\Gamma,\SL(r,\CC)) \times \CC^*)/ \mu_r = X(\Gamma,\SL(r,\CC))\times_{\mu_r} \CC^*\,.\qedhere
\end{align*}
\end{proof}

\subsection*{Distinguished component}
For a hyperbolic knot group, up to complex conjugation
there exists a unique one-dimensional component
$X_0\subset X(\Gamma,\PSL(2,\CC))$ which contains the character of the holonomy representation.
The holonomy representations lifts to a representation $\rho\co\Gamma\to\SL(2,\CC)$, and by composing any  
lift with the rational, irreducible, $r$-dimensional representation
$\mathrm{Sym}^{r-1}\co\SL(2,\CC)\to\SL(r,\CC)$ we obtain an irreducible representation
$\rho_r\co\Gamma\to\SL(r,\CC)$. It follows from \cite{MenalP} that 
$\chi_{\rho_r}\in X(\Gamma,\SL(r,\CC))$ is a smooth point contained in a unique $(r-1)$-dimensional 
component of $X(\Gamma,\SL(r,\CC))$. We will call such a component a \emph{distinguished} component of $X(\Gamma,\SL(r,\CC))$.
For \emph{odd} $r$, as $\mathrm{Sym}^{r-1}\co\SL(2,\CC)\to\SL(r,\CC)$ factors through $\PSL(2,\CC)$, there is a unique distinguished component 
in $X(\Gamma,\SL(r,\CC))$ up to complex conjugation.

\subsection*{Totally reducible representations}
%%%%%%%%%%%%%%%%%%%%%%%%%%%%%%%%%%%%%%%%

Totally reducible representations are representations which
split as a direct sum of one-dimensional representations. In particular they are representations 
of the abelianization of a knot group $\Gamma$, which is $\ZZ$.
Thus the restriction of a totally reducible representation
to the commutator subgroup is trivial and it only depends of the image of a meridian, that is a diagonal matrix.

If the image of a meridian is  $\diag (\lambda_1,\ldots,\lambda_r)$, then the 
space of parameters is $(\sigma_1,\ldots,\sigma_r)$,
where $\sigma_i$ is the $i$-th elementary symmetric polynomial on 
$\lambda_1,\ldots,\lambda_r$ 
(hence $\sigma_r=1$ for $\SL(r,\CC)$).
Thus for any 
knot group $\Gamma$
 \begin{align}
 X_\mathrm{TR}(\Gamma,\SL(r,\CC)) &= \CC^{r-1} \, , \notag\\
 X_\mathrm{TR}(\Gamma,\GL(r,\CC)) &= \CC^{r-1}\x \CC^* \, , \label{eq:tot-red}\\
 X_\mathrm{TR}(\Gamma,\PGL(r,\CC)) &= (\CC^{r-1})/\mu_r \,  ,\notag
 \end{align}
where $\varpi\cdot(\sigma_1,\ldots,\sigma_{r-1}) = 
(\varpi \sigma_1, \varpi^2\sigma_2 ,\ldots,\varpi^{r-1}\sigma_{r-1})$.

Now (and in the rest o the paper) we move to the specific case where $\Gamma$ is the figure eight knot group.
As the group $\Gamma$ is generated by $S$ and $T$ (Presentation \eqref{eq:pres-2bridge}), 
$X(\Gamma,\SL(3,\CC))$
can be viewed as a subvariety of the variety of characters of a free group generated by $S$ and $T$. In addition, since $S$ and $T$ 
are conjugate in $\Gamma$,
the results of Lawton's and Will on the variety of characters of free groups yield \cite{Lawton0,Will}:

\begin{proposition} \label{prop:parametersLawton}
For $ \Gamma$ the  figure eight knot group,  the character variety  $X(\Gamma,\SL(3,\CC))$ embeds into $\CC^8$.
The embedding is given by the  parameters  $(y, \bar y, z, \bar z, \alpha, \bar\alpha, \beta, \bar\beta )$ 
in \eqref{eq:parameters}.
% % \begin{alignat}{4}
% %   y(\chi) &=\chi(S), \ & \bar{y}(\chi)&=\chi(S^{-1}),\ & z(\chi)&=\chi(ST),\ & \bar{z}(\chi)&=\chi(T^{-1}S^{-1}),\notag\\
% % \beta(\chi)&=\chi(TS^{-1}), \ & \bar\beta(\chi) &=\chi(ST^{-1}), \ &
% % \alpha(\chi)&=\chi([S,T]), \ & \bar\alpha(\chi) &=\chi([T,S])\,.
% % \end{alignat}
\end{proposition}

Of course the previous proposition applies to any two-bridge knot group, as it has a presentation
similar to \eqref{eq:pres-2bridge} with two generators represented by meridian curves.

%%%%%%%%%%%%%%%%%%%%%%%%%%%%%%%%%%%%%%%%
\section{Representations in $\SL(2,\CC)$, $\GL(2,\CC)$, and $\PGL(2,\CC)$}
%%%%%%%%%%%%%%%%%%%%%%%%%%%%%%%%%%%%%%%%
\label{section:SL2}

Let  $\Gamma$ be the knot group of the figure eight, in this section we 
analize the space of representations 
$X(\Gamma, G)$ for  $G=\SL(2,\CC)$, $\GL(2,\CC)$ and $\PGL(2,\CC)$.
Reducible representations are totally reducible, hence they have been described
in Section~\ref{sec:general}, and
we discuss next \emph{irreducible} representations.

To understand the irreducible representations of the figure eight knot group into
$\SL(2,\CC)$, we follow \cite{Porti}.
For $G= \SL(2,\CC)$ or $\GL(2,\CC)$, %or $\PGL(2,\CC)$, 
let $\chi\in X(\Gamma,G)$ be a character. 
%, $T=\rho(t)\in G$.
%We will usually say that $(A,B,T)$ is a representation of $X(\Gamma,G)$ when referring to $\rho$. 
Consider the restriction map
 $$
 \res\co X(\Gamma,G)\to X(F_2,  \SL(2,\CC)).
 $$
We use Fricke coordinates for  $X(F_2,  \SL(2,\CC))$, given as
$$
x_1(\chi)=\chi (a), \qquad 
x_2(\chi)=\chi (b), \qquad 
x_3(\chi)=\chi (ab), \qquad 
$$
which define an isomorphism $X(F_2,\SL(2,\CC)) \cong \CC^3$.
In those coordinates, conjugation by $t$ induces a transformation given by:
 $$
 (x_1,x_2,x_3)(\chi)\mapsto (\chi(ab),\chi(bab),\chi(ab^2ab))= 
 (x_3,x_2x_3-x_1,x_2x_3^2 -x_1 x_3-x_2).
 $$
Here we have used the basic identities for $Y,Z\in\SL(2,\CC)$,
\begin{equation}
 \label{eqn:basicSL2}
 \tr (YZ)=\tr (Y) \tr (Z)-\tr (YZ^{-1}), \quad \tr (YZ)= \tr (ZY), \quad \tr (Y^{-1})=\tr( Y)\, .
\end{equation}
%Thus
%\begin{align*}
% \res(X(\Gamma, G)) & =\{(x_1,x_2,x_3)\in \CC^3\mid x_1=x_3,\ x_2= x_2x_3-x_1,\ x_3=x_2x_3^2 -x_1 x_3-x_2\} \\
%  &\cong \{(x_1,x_2)\in \CC^2\mid x_1 x_2=x_1+x_2\}   .
%\end{align*}
%
%Namely we get a curve, the point with coordinates $(x_1,x_2)=(2,2)$ stands for the projection of reducible 
%representations (trivial in $F_ 2$),
%and the rest of the curve corresponds to irreducible representations.
%Note that the curve is rational, parametrized by $x_2=x_1/(x_1-1)$, $x_1\in \CC -\{1\}$.
%
%
%To get   $X(\Gamma,  \SL(2,\CC))$ we introduce  a third variable:
% $$
% y(\chi)=\chi(t)=\chi( ta)=\chi(tb)=\chi(tab),
% $$
%(it is straightforward that $t$, $ta$, $tb$ and $tab$ are conjugate elements).
%By applying the identities \eqref{eqn:basicSL2} to the trace of $b= [a,t^{-1}]$, 
%we get $x_2=x_1^2-y^2(x_1-2)-2$. Hence there are two components:
% $$
% x_1=2,  \  x_2=2, \ y\in \CC,
% $$
%corresponding to reducible representations, and
% \begin{equation} \label{eqn:y}
% x_1+x_2=x_1x_2, \  y^2=x_1+x_2+1,
%  \end{equation}
%corresponding to irreducible ones (except at the points  
%$x_1=x_2=2$, $y=\pm\sqrt{5}$, where the two curves intersect,
%that are reducible representations). 
Thus
\begin{align*}
 \res(X(\Gamma, G)) & =\{(x_1,x_2,x_3)\in \CC^3\mid x_1=x_3,\ x_2= x_2x_3-x_1,\ x_3=x_2x_3^2 -x_1 x_3-x_2\} \\
  &\cong \{(x_1,x_2)\in \CC^2\mid x_1 x_2=x_1+x_2\} \\
  &\cong  \{(x_1,x_2)\in \CC^2\mid 
  (x_1-1)(x_2-1)=1\}\cong\CC-\{1\} .
\end{align*}
%\comment{M: it seem's to me that is simpler  to state directly the isomorphism to 
%$\CC\setminus\{1\}$. Moreover, we obtain that $1/(x_1-1)$ is a regular function which simplifies the argument.}\\
The point with coordinates $(x_1,x_2)=(2,2)$ stands for the restriction of reducible 
representations to $F_ 2$,
and the rest of the curve corresponds to irreducible representations.
%Note that the curve is rational, parametrized by $x_2=x_1/(x_1-1)$, $x_1\in \CC -\{1\}$.

To get   $X(\Gamma,  \SL(2,\CC))$ we introduce  a third variable:
 $$
 y_0(\chi)=\chi(t)=\chi( ta)=\chi(tb)=\chi(tab),
 $$
(it is straightforward that $t$, $ta$, $tb$ and $tab$ are conjugate elements). The group
$\Gamma$ is generated by $t=S$ and $T=a^{-1}t a$. Notice that
$b=[a,t^{-1}]$ is conjugate to $TS^{-1}$, and therefore
coordinates for $X(\Gamma,\SL(2,\CC))$ are given by $x_1$, $x_2$ and $y_0$.
By applying the identities \eqref{eqn:basicSL2} to the trace of $b= [a,t^{-1}]$, 
we get $x_2=x_1^2-y_0^2(x_1-2)-2$ and $x_2 = x_1/(x_1-1)$. Hence 
\begin{align*}
0&= x_{1}^{3} -  x_{1}^{2} (y_0^{2} -  1) + 3 x_{1} (y_0^{2} - 3)  - 2 (y_0^{2} -2 )\\
&= (x_1-2) (x_1-1) ( x_{1}+1 - y_0^{2} + \frac{x_{1}}{x_1-1}) \,.
\end{align*}
Hence
there are two components:
 $$
 x_1=2,  \  x_2=\frac{x_{1}}{x_1-1}=2, \ y_0\in \CC,
 $$
corresponding to reducible representations, and
 $(x_1-1)(x_2-1)=1$, $y_0^2=x_1+x_2+1$,
corresponding to irreducible ones.
% (except at the points  
%$x_1=x_2=2$, $y_0=\pm\sqrt{5}$, where the two curves intersect,
%that are reducible representations). 

% \begin{remark}
% The curve (\ref{eqn:y}) is isomorphic to the plane curve
%  $$
%  y_0^2= \frac{x_1^2+x_1 -1}{x_1-1}   ,
%  $$
% which is a double cover of $\CC-\{1\}$ ramified at 
% $\frac{-1\pm  \sqrt5}2$, $1$, and $\infty$. Hence it is an elliptic curve $E$ 
% with two points removed and  with $j$-invariant 
%  $$
%   j(E)=256 \frac{(1-r+r^2)^3}{r^2(1-r)^2}=1728 \, \frac{32}{5}\, ,
%   $$
% where $r=\frac12 (7 +3\sqrt5)$ is the double ratio of the ramification points.
% This $E$ intersects the other component $\CC$ in two points.
% \end{remark}

\begin{proposition} \label{prop:SL2}
 The character variety $X(\Gamma,  \SL(2,\CC))$ has two irreducible components,
written as follows, in terms of the coordinates $x_1=\chi(a)$, $x_2=\chi(b)$ and $y_0 =\chi(t)$:
  \begin{itemize}
 \item The component corresponding to reducible representations: 
 $x_1=2,  x_2=2, y_0\in \CC$.
 \item The component corresponding to irreducible representations: 
 \begin{equation} \label{eqn:y}
  (x_1-1)(x_2-1)=1,\  y_0^2=x_1+x_2+1\,.
 \end{equation}
 \end{itemize}
They intersect transversally in the two points: $x_1=x_2=2$, $y=\pm\sqrt{5}$.
\end{proposition}
%\comment{M: if you do not like the presentation with $(x_1-1)(x_2-1)=1$ we can change it.}

From this, to get $\PGL(2,\CC)$-characters, we have to quotient by
$\mu_2=\{\pm 1\}$ acting on $y_0$, that is 
by involution $y_0\mapsto -y_0$ (see Section~\ref{sec:general} and \cite{MGA}). 
%This produces two components $\CC/\mu_2 \cong \CC$, consisting
%of reducible representations,  and $E/\mu_2$,  which is parametrized by $x_1 \in \CC-\{1\}$. 
%These two components intersect in a single point.

\begin{proposition} \label{prop:PGL2}
 The character variety $X(\Gamma,  \PGL(2,\CC))$ has two irreducible components,
written as follows, in terms of the coordinates $x_1=\chi(a)$, $x_2=\chi(b)$ and $z_0=y_0^2 =\chi(t)^2$:
  \begin{itemize}
 \item The component corresponding to reducible representations: 
 $x_1=2,  x_2=2, z_0\in \CC$.
 \item The component corresponding to irreducible representations: 
\[(x_1-1)(x_2-1)=1,\  z_0=x_1+x_2+1\,,\]
which is isomorphic to $\CC-\{1\}$.
 \end{itemize}
They intersect in one point: $x_1=x_2=2, z_0=5$.
\end{proposition}
%\colorbox{yellow}{M: I do not know if it makes sense to put the $\PGL(2)$-proposition here.
%Do we need it ?}

To get the $\GL(2,\CC)$-representations,  recall from Lemma~\ref{lem:GL} that
 $$
 X(\Gamma,\GL(2,\CC)) = (X(\Gamma,\SL(2,\CC)) \x \CC^*)/\mu_2 \, .
  $$
This algebraic set has two components: the one consisting of characters of reducible representations is isomorphic
to $(\CC\x \CC^*)/\mu_2$, and the one containing characters of irreducible representations that is
isomorphic to $(E\x \CC^*)/\mu_2$, 
where $E$ is the curve defined by (\ref{eqn:y}). 
The action of $\mu_2$ on $\CC\x \CC^*$ generates the equivalence $(y_0,\lambda)\sim (-y_0,-\lambda)$. 
The ring of invariant functions of this action is generated by $u=y_0^2$, $v=y_0\lambda$, $w=\lambda^2$, and
the algebraic relations between these functions are generated by $uw=v^2$. 
The variable $u$ can be eliminated since $1/w$ is a regular function on $\CC^*$. 
We obtain that $(\CC\x \CC^*)/\mu_2\cong \CC\x \CC^*$.
It coincides with the component $X_\mathrm{TR}(\Gamma,\GL(2,\CC))$ 
(see Equation~\eqref{eq:tot-red}).

The product $E\x \CC^*$  is parametrized by $(x_1,y_0,\lambda)$, satisfying
the equations 
(\ref{eqn:y}). In order to obtain $(E\x \CC^*)/\mu_2$ we have to identify 
$(x_1,y_0,\lambda) \sim (x_1,-y_0,-\lambda)$. Now, the ring of invariant functions of this action 
is generated by $u=y_0^2$, $v=y_0\lambda$, $w=\lambda^2$ and $x_1$, and
the algebraic relations between these functions are generated by $uw=v^2$. Hence
$(E\x \CC^*)/\mu_2$ is isomorphic to
\[ \{ (v,w,x_1)\in \CC\x\CC^*\x(\CC-\{1\})\mid \frac{x_1^2 +x_1-1}{x_1-1} w = v^2\}\,.\]
The intersection of the two components is given by introducing the additional equation $x_1=2$,
and therefore:
$$(E\x \CC^*)/\mu_2\cap (\CC\x \CC^*)/\mu_2=\{ (v,v^2/5,2)\mid v\in\CC^*\}\cong\CC^*.$$

Notice that for a representation $\rho_2\co\Gamma\to\GL(2,\CC)$ with character 
$\chi_2 := \chi_{\rho_2}\in X(\Gamma,\GL(\CC))$ we have:
\[ v(\chi_2)=\tr(\rho_2(t))\in\CC ,\quad w(\chi_2) = \det(\rho_2(t))\in\CC^*,\quad 
x_1(\chi_2)=\tr(\rho_2(a))\in\CC-\{1\}\,.\]

\begin{proposition} \label{prop:GL2}
 The character variety $X(\Gamma,  \GL(2,\CC))$ has two irreducible 2-dimensional components.
%written as follows, in terms of the coordinates 
%$(v,w,x_1)\in  \CC\x\CC^*\x(\CC-\{1\})$:
%$x_2(\chi_\rho)=\tr (B)$, 
%$v =\tr(\rho(t))$, and $w=\det (\rho(t))$:
More precisely,
  \begin{itemize}
 \item The component $X_\mathrm{TR}(\Gamma,\GL(2,\CC))$ contains only characters of reducible representations is isomorphic to $\CC\x\CC^*$:  
% $x_2=2$, 
\[ X_\mathrm{TR}(\Gamma,\GL(2,\CC)) \cong 
\{ (v,w,x_1)\in  \CC\x\CC^*\x(\CC-\{1\})\mid x_1= 2 \} \cong \CC\x\CC^*\,.\]
 \item The component $X_2$ which contains characters of irreducible representations: 
\[  X_2 \cong 
\{ (v,w,x_1)\in  \CC\x\CC^*\x(\CC-\{1\})\mid \frac{x_1^2 +x_1-1}{x_1-1} w = v^2\}\,.\]
 \end{itemize}
The intersection $X_0\cap X_\mathrm{TR}(\Gamma,\GL(2,\CC))$ is isomorphic to $\CC^*$: 
\[ \{ (v,w,2)\in  \CC\x\CC^*\x(\CC-\{1\})\mid 5 w = v^2\}\cong\CC^*\,.\]
\end{proposition}

%%%%%%%%%%%%%%%%%%%%%%%%%%%%%%%%%%%%%%%%%%%%%%%%%%%%%%%%%%%%%%%%%%%%%
\section{Reducible representations  in $\SL(3,\CC)$} \label{sec:reducible}

We start by describing the
\emph{totally reducible} characters already given in
Section~\ref{sec:general}, Equation~\eqref{eq:tot-red},
with the coordinates \eqref{eq:parameters}:
$$
\alpha=\bar\alpha=\beta=\bar\beta=3, \ (y,\bar{y})\in \CC^2, \ z = y^2-2\bar{y},\textrm{ and } 
\bar{z} = \bar{y}^2-2y\, .
$$
Here we have used that, by the Cayley-Hamilton theorem, 
for every $A\in\SL(3,\CC)$ the equality
\( \tr(A^2) = \tr^2(A) - 2\tr(A^{-1})\) holds.

Now we move to \emph{partially reducible} representations, that is,
representations that are a direct sum of a $2$-dimensional representation
and a $1$-dimensional representation.
Let $\rho_2\co\Gamma\to\GL(2,\CC)$ be irreducible. Then
$\rho = \rho_2 \oplus (\det\,\rho_2)^{-1}$ is partially reducible, and if
$(v,w,x_1)$ denote the coordinates of the character $\chi_2 := \chi_{\rho_2}$ then
the coordinates of $\chi_\rho$ are functions of $\chi_2=(v,w,x_1)$. More precisely:
\[
  \alpha(\chi_2)  =\bar\alpha(\chi_2)=x_1+1, \quad
   \beta(\chi_2) =\bar\beta(\chi_2)=x_2+1, \quad
  y(\chi_2)=v+\frac{1}{w}, \quad \bar{y}(\chi_2) =\frac{v}{w}+w. %\\
% z(\chi_\rho)&=v^2-w x_2+ \frac{1}{w^2} ,\ & 
% \bar{z}(\chi_\rho)&= \frac{v^2}{w}- x_2+ \frac{1}{w^3}
\]
%
%\begin{alignat*}{2}
%  \alpha(\chi_\rho) & =\bar\alpha(\chi_\rho)=x_1+1, \quad  & \beta(\chi_\rho)&=\bar\beta(\chi_\rho)=x_2+1, \\
%  y(\chi_\rho)&=v+\frac{1}{w}, \ & \bar{y}(\chi_\rho)&=\frac{v}{w}+w , %\\
%% z(\chi_\rho)&=v^2-w x_2+ \frac{1}{w^2} ,\ & 
%% \bar{z}(\chi_\rho)&= \frac{v^2}{w}- x_2+ \frac{1}{w^3}
%\end{alignat*}
In order to calculate $z(\chi_2)$, we will use that for all $Y,Z\in\GL(2,\CC)$ the identities
\[
\tr(YZ) = \tr(Y)\tr(Z) - \det(Y)\tr(Y^{-1}Z)  \qquad \text{ and } \qquad \tr(Y^{-1})\det(Y) =\tr(Y)
\]
hold. This gives
\[
z(\chi_2) =v^2-w x_2 + \frac{1}{w^2} \qquad \text{ and } \qquad 
\bar{z}(\chi_2)= \frac{v^2-wx_2}{w^2} + w^2.
\]
Now, we have 
\[ \frac{v^2}{w} = x_1+x_2 +1\]
and hence we obtain
\begin{alignat}{2} \label{eqn:PR}
  \alpha(\chi_2)  =\bar\alpha(\chi_2)&=x_1+1, \quad &
   \beta(\chi_2) =\bar\beta(\chi_2)&=x_2+1, \notag\\
  y(\chi_2)&=v+\frac{1}{w}, \quad & \bar{y}(\chi_2) &=\frac{v}{w}+w. \\
z(\chi_2) & = w (x_1+1)+ \frac{1}{w^2},\quad &
 \bar{z}(\chi_2) &= \frac{x_1+1}{w}+ w^2\,.\notag
 \end{alignat}
It follows that the component $X_\mathrm{PR}(\Gamma,\SL(3,\CC))$ of partially 
reducible characters is parametrized by the component $X_2 = 
(E\times \CC^*)/\mu_2 \subset X(\Gamma,\GL(2,\CC))$ of irreducible characters.
If  $\rho_2,\rho'_2\in R(\Gamma,\GL(2,\CC))$ are two semisimple representations then
$\rho =\rho_2 \oplus (\det\,\rho_2)^{-1}$ and $\rho' =\rho'_2 \oplus (\det\,\rho'_2)^{-1}$
determine the same character in $X(\Gamma,\SL(3,\CC))$ if and only if they are conjugate. 
Here we have used that each point  of $X(\Gamma,\SL(3,\CC))$ is the character of a semi-simple 
representation which is unique up to conjugation \cite{LM}. Hence, two characters 
$\chi_2,\,\chi_2'\in X_2$ can only give the same character of $X(\Gamma,\SL(3,\CC))$ 
if $\rho_2$ and $\rho'_2$ are reducible i.e.\ if for the corresponding parameters
$(x_1,v,w)$ and $(x'_1,v',w')$ of $\chi_2$ and $\chi'_2$ respectively the equations
$x_1=x_1'=2$, $v^2 = 5w$, and  $(v')^2 = 5w'$ hold. 
Therefore, if $\chi_2$ and $\chi_2'\in X_2$ give the same character of $X(\Gamma,\SL(3,\CC))$ 
then $y(\chi_2)=y(\chi'_2)$ and $\bar y(\chi_2)=\bar y(\chi'_2)$. This is equivalent to
\[ v +  \frac{5}{v^2} =  v' +  \frac{5}{(v')^2} \ \text{ and }\ 
\frac{5}{v} +  \frac{v^2}{5} =  \frac{5}{v'} +  \frac{(v')^2}{5} \,.\]
 If $v=v'$ then $w=w'$, and $\chi_2=\chi_2'$ follows. If $v\neq v'$ then $(vv')^3=125$, and there
 are exactly three pairs of reducible characters which map to the same the character in 
 $X(\Gamma,\SL(3,\CC))$:
 \begin{align*}
 \left(\frac{5}{2} \pm \frac{\sqrt5}{2},\frac{3}{2} \pm \frac{\sqrt5}{2},2 \right)& \mapsto
 (4,4,8,8,3,3,3,3)\,,\\
  \left(\varpi\left(\frac{5}{2} \pm \frac{\sqrt5}{2}\right),\varpi^2\left(\frac{3}{2} \pm \frac{\sqrt5}{2}\right),2\right)& \mapsto
\varpi\cdot (4,4,8,8,3,3,3,3)\,,\\
  \left(\varpi^2\left(\frac{5}{2} \pm \frac{\sqrt5}{2}\right),\varpi\left(\frac{3}{2} \pm \frac{\sqrt5}{2}\right),2\right)& \mapsto
\varpi^2 \cdot (4,4,8,8,3,3,3,3)\,.
 \end{align*}

\begin{proposition} \label{prop:SL3-reducibles}
 The locus of reducible representations of the
character variety $X(\Gamma,  \SL(3,\CC))$ consists of two irreducible components:
  \begin{itemize}
 \item The component $X_\mathrm{TR}:=X_\mathrm{TR}(\Gamma,  \SL(3,\CC))$ contains only characters of totally reducible representations and it is isomorphic to $\CC^2$:
 \[ 
 \big\{ (y,\bar y,z,\bar z,\alpha,\bar \alpha,\beta,\bar\beta)\in\CC^8 \mid
 \alpha=\bar\alpha= \beta=\bar\beta=3,\; z=y^2-2\bar y,\; \bar z= \bar y^2-2 y\big\}\,.\]
 \item The component $X_\mathrm{PR}:=X_\mathrm{PR}(\Gamma,  \SL(3,\CC))$ contains only characters of reducible representations is parametrized by the component 
 $X_2  \subset X(\Gamma,\GL(2,\CC))$ (see Proposition~\ref{prop:GL2}).
A parametrization is given by:
  \begin{alignat*}{2}
    \alpha(v,w,x_1)=\bar\alpha(v,w,x_1)&= x_1+1, \quad & 
    \beta(v,w,x_1)=\bar\beta(v,w,x_1)&=\frac{x_1}{x_1-1} +1 ,\\ 
  y(v,w,x_1)=v+\frac1w,\quad &  \bar{y}(v,w,x_1)=w+\frac{v}{w}, &
z(v,w,x_1)=w\, \alpha+\frac{1}{w^2},&\quad 
\bar z(v,w,x_1)=\frac{\alpha}{w}+w^2\,.
  \end{alignat*}  
 % \begin{align*}
%   & \alpha=\bar\alpha , \ \beta=\bar\beta ,\ 
%   \alpha\beta-2\alpha-2\beta+3=0 ,\  y=v+\frac1w, \ \bar{y}=w+\frac{v}{w}, \\
%  &  z=w \alpha+\frac1{w^2}, \ \bar z=\frac{\alpha}{w}+w^2, \  (\alpha+\beta-1)w= v^2, 
%  \end{align*}
%  where $v\in\mathbb{C}$ and  $ w\in\mathbb{C}^*$.
 \end{itemize}
 The component $X_\mathrm{PR}$ is smooth except at the three points
 $\mu_3\cdot(4,4,8,8,3,3,3,3)$ which are contained in the intersection 
 $X_\mathrm{TR}\cap X_\mathrm{PR}$. Moreover,  $X_\mathrm{TR}\cap X_\mathrm{PR}$
is isomorphic to the plane curve with equations: 
$\alpha=\bar\alpha=\beta=
\bar\beta=3$, $64  - 28 y \bar{y} - y^2 \bar{y}^2 + 5(y^3 + \bar{y}^3)=0$,
$z=y^2-2\bar y$, $\bar z=\bar y^2-2 y$.
\end{proposition}

\begin{remark}
It can be checked that the parametrization of $X_\mathrm{PR}$ is an immersion and that the singularities are nodal, i.e.\ 
two branches of the parametrization are smooth and intersect transversely at each of the three singular points
 $\mu_3\cdot(4,4,8,8,3,3,3,3)$.
\end{remark}

To finish the section, we describe the set of reducible characters in $X(\Gamma,  \SL(3,\CC))$
that lie in the closure of the components of irreducible characters. Recall that the set of
reducible characters $X_{red}=X_{\mathrm{TR}}\cup X_{\mathrm{PR}}$ is Zariski-closed and its 
complement $X_\mathrm{irr}$ is Zariski-open.

\begin{lemma}
\label{lemm:redtoirreduc}
The set $X_{red}(\Gamma,\SL(3,\CC))\cap\overline{ X_{irr}(\Gamma,\SL(3,\CC)) }$
is parametrized by a singular curve $\mathcal{C}\subset X_2$ given by
\[
\mathcal{C} = \big\{ (w,x_1)\in  \CC^*\x(\CC-\{1\})\mid 
w^6-2w^3\frac{2x_1^2+x_1-1}{x_1-1} +1
\big\}\,.
\]
More precisely, the curve $\mathcal{C}$ has exactly three singular points $\mu_3\times\{0\}$.
The parametrization is given by restricting the parametrization of Proposition~\ref{prop:SL3-reducibles} 
to $\mathcal{C}$ i.e. by substituting $v = (w^3+1)/2w$.
%is the set with coordinates (in the notation of Proposition~\ref{prop:SL3-reducibles})
% $\alpha=\bar\alpha$,  $\beta=\bar\beta$, $\alpha\beta-2\alpha-2\beta+3=0$,
% $w^3-2 v w+1=0$, $(\alpha+\beta-1)w= v^2$.
%  
%  
%  
%  
%  $ y  \bar{y}= 3 A+1$, $ y^3+\bar{y}^3=2(A^2+8 A-1)$, where $A=\alpha+\beta-1$.
In addition, all points of $X_{red}(\Gamma,\SL(3,\CC))\cap\overline{ X_{irr}(\Gamma,\SL(3,\CC)) }$ are smooth points of 
$X_\mathrm{red}(\Gamma,\SL(3,\CC))$ and $\overline{ X_\mathrm{irr}(\Gamma,\SL(3,\CC)) }$ 
respectively, with the exceptions of the six points
$\mu_3\cdot(2,2,2,2,1,1,1,1)$ and $\mu_3\cdot(4,4,8,8,3,3,3,3)$. 
\end{lemma}

\begin{proof}
A reducible and semisimple representation  $\rho\co\Gamma\to \SL(3,\CC)$ 
with character $\chi_\rho$
can be written as $\rho=({\phi}\otimes\varrho)\oplus \phi^{-2}$,
with representations $\phi\co\Gamma\to\CC^*$ and $\varrho\co\Gamma\to \SL(2,\CC)$. Namely,
$\rho(\gamma)=\diag(\phi(\gamma)\varrho(\gamma),\phi(\gamma)^{-2})$, $\forall\gamma\in\Gamma$.
Thus
\begin{equation}
\label{eqn:yandz}
% y=\chi_{\rho}(t)=\phi(t)\chi_{\varrho}(t)+\phi(t)^{-2}\quad\textrm{ and }\quad
% \bar{y}=\chi_{\rho}(t^{-1})=\phi(t)^{-1}\chi_{\varrho}(t)+\phi(t)^2.
v=\chi_{\varrho}(t) \phi(t)\quad \textrm{ and }\quad w= \phi(t)^2\, .
\end{equation}
% If $v$ and $w$ are as in  Proposition~\ref{prop:SL3-reducibles}, then 
% $v=\chi_{\varrho}(t) \phi(t)$ and $w= \phi(t)^2$.
By \cite[Theorem~1.3]{HeusenerP},  if  $\chi_\rho\in    \overline{X_{irr}(\Gamma,\SL(3,\CC))}$
is of type (2,1), then $\xi=\phi(t)$
satisfies
\begin{equation}
 \label{eqn:zeroAlex}
\xi^6-2\chi_\varrho(t) \xi^3+ 1=0.
 \end{equation}
This uses that the Alexander polynomial twisted by $\varrho$ is  
$\Delta^{\varrho}(x)=x^2-2\chi_\varrho(t) x+1$.
Kitano \cite{Kitano} has computed  $ \Delta^{\varrho}(1)=2-2\chi_\varrho(t)$, and
the same computation yields 
$ \Delta^{\varrho}(x)$ as follows: using the fibration of the figure eight
knot, we know that $ \Delta_{\varrho}(x)=x^2-\delta(\chi) x+1$ for some
regular function $\delta\co X_{irr}(\Gamma,\SL(2,\CC))\to \CC$; in addition $\delta$ is determined from  
$ \Delta^{\varrho}(1)=2-2\chi_\varrho(t)$. This implies that
$\xi^3+\xi^{-3}=2\chi_\varrho(t)$.
This condition is necessary for a partially reducible representation $\rho$, i.e.~when $\varrho$ is irreducible.
The representation $\varrho$ is reducible precisely when $ \alpha=\beta=3$ and the 
condition is also necessary in this case  if we can show that
\begin{equation}
\label{eqn:aprox}
X_\mathrm{TR}(\Gamma,\SL(3,\CC))\cap \overline{X_{irr}(\Gamma,\SL(3,\CC))}\subset
\overline{X_\mathrm{PR}(\Gamma,\SL(3,\CC))}.
\end{equation}
To prove \eqref{eqn:aprox}, the discussion  in \cite{HPS} using twisted cohomology yields that
the ratio between two eigenvalues of $\rho(t)$ is a root  of
the untwisted Alexander polynomial $t^2-3t+1$. By
\cite{HPS}, this is also a sufficient condition for a $2\times 2$
block of $\rho$ being approximated by irreducible representations
in $\GL(2,\CC)$,
because it is a simple root.
Using \eqref{eqn:yandz} and  \eqref{eqn:zeroAlex}, we get
$
w^3-2 v w+1=0$, equivalently  $v=\frac{w^3+1}{2w}$.
By replacing the value of $v$ in   $\frac{x_1^2 +x_1-1}{x_1-1} w = v^2$
we get the equation of the lemma.
% 
% and that $A=\chi_\varrho(t)^2$,
% we deduce the formulas for $y \bar{y}$ and $y^3+\bar{y}^3$. It is easy to chech that 
% $\alpha$, $\beta$, $y$, and $\bar{y}$ determine the values of the (semisimple
% reducible) representation up to conjugation.
 
Finally, we notice that when $(\alpha,\beta)\neq (1,1), (3,3)$,
\cite[Corollary 8.9]{HeusenerP} applies and corresponding characters are smooth points of 
$X_\mathrm{PR}$ and $\overline{X_\mathrm{irr}}$ respectively,
as the corresponding roots of the twisted Alexander polynomial are simple.
\end{proof}

From Proposition~\ref{lemm:redtoirreduc} it follows that 
 \begin{equation}\label{eq:fibre}
 w^6-2w^3\alpha\frac{2\alpha-3}{\alpha-2}+1=0
 \end{equation}
on $X_{red}(\Gamma,\SL(3,\CC))\cap \overline{X_{irr}(\Gamma,\SL(3,\CC))}$. Thus we get:

\begin{corollary} \label{cor:fibre}
The fiber of the projection 
$$
  X_{red}(\Gamma,\SL(3,\CC))\cap \overline{X_{irr}(\Gamma,\SL(3,\CC))}\to \{(\alpha,\beta)\in \CC^2\mid (\alpha-2)(\beta-2)=1\}\cong
  \CC-\{2\}
$$
has six points except when $(\alpha,\beta)= (1,1)$, $(3,3)$, or 
$\left(\frac{1\pm\sqrt{5}}{2}, \frac{1\mp\sqrt{5}}{2}\right)$, where it has three points.
\end{corollary}

\begin{proof} 
First notice that $\beta =(2\alpha-1)/(\alpha-2)$.
Each character in $ X_{red}(\Gamma,\SL(3,\CC))\cap \overline{X_{irr}(\Gamma,\SL(3,\CC))}$
determines a unique parameter $w$ with the exception of the three characters with coordinates 
$\mu_3\cdot(4,4,8,8,3,3,3,3,3)$. To such a singular point correspond two values of 
$w\in  \mu_3\cdot \left(\frac{3}{2} \pm \frac{\sqrt5}{2}\right)$. In this case 
$w^3 = 9\pm 4\sqrt{5}$ and $\alpha =\beta=3$ follows from~\eqref{eq:fibre}.

The discriminant of~\eqref{eq:fibre} is $4(\alpha^2 - \alpha - 1)(\alpha - 1)^2$ and hence it 
vanishes if and only if $\alpha = \beta =1$ or 
$\alpha=\frac{1\pm\sqrt{5}}{2}$ and $\beta = \frac{1\mp\sqrt{5}}{2}$.
\end{proof}

%\comment{M: This section is one of the most complicate things I have ever written down\ldots
\begin{remark} 
We will see later in Proposition~\ref{prop:smooth} that $\mu_3 \cdot (2,2,2,2,1,1,1,1)$ are the only singular points of 
the components of
$\overline{ X_{irr}(\Gamma,\SL(3,\CC)) }$.
We already saw that  $\mu_3 \cdot (2,2,2,2,1,1,1,1)$ are smooth points of $X_\mathrm{PR}$. Hence,
$\mu_3 \cdot (4,4,8,8,3,3,3,3)$ are smooth points on $\overline{X_\mathrm{irr}(\Gamma,\SL(3,\CC)) }$ and $X_\mathrm{TR}$,
and they are singular on $X_\mathrm{PR}$.
\end{remark}

%%%%%%%%%%%%%%%%%%%%%%%%%%%%%%%%%%%%%%%%%%%%%%%%%%%%%%%%%%%%%%%%% 
 \section{Restriction to the fibre} \label{sec:Lawton}
%%%%%%%%%%%%%%%%%%%%%%%%%%%%%%%%%%%%%%%%%%%%%%%%%%%%%%%%%%%%%%%%%

In this section and up to  Section~\ref{sec:irredPGL3GL3} we work with 
Presentation~\eqref{eq:pres-fiber}, corresponding to the fibration 
over the circle with fibre  a punctured torus. In particular, the group of the fibre
 $F_2$ is the free group of rank 2 generated by  $a$ and $b$.
To compute $X_{irr}(\Gamma, G)$ for $G=\SL(3,\CC)$, 
$\PGL(3,\CC)$, or $\GL(3,\CC)$, 
we look at the restriction: 
$$
 \res :X(\Gamma,G)\to X(F_2,\SL(3,\CC)),
$$
whose image does not depend on $G$.
For $X(F_2,\SL(3,\CC))$, we use Lawton's description in \cite{Lawton0}. 
According to it, there is a two fold branched
covering 
$$
 \pi:X(F_2,\SL(3,\CC))\to \CC^8,
$$ 
where the coordinates of $\CC^8$ are the traces of 
\begin{equation}
 \label{eq:8}
 a,\ a^{-1},\ b,\ b^{-1},\ ab,\  b^{-1} a^{-1}, \ ab^{-1}, \ a^{-1} b.  
\end{equation}
 The branched covering comes from a ninth coordinate, which is the trace of the commutator
$$
[a,b]=a b a^{-1} b^{-1}.
$$
This trace satisfies a polynomial equation
$$
x^2-P x+Q=0,
$$
where $P$ and $Q$ are polynomials on the first eight variables
(see \cite{Lawton0} for the expression of $P$ and $Q$). 
The solutions are precisely
the trace of $[a,b]$ and the trace of its inverse. 
 
First of all we reduce the eight coordinates to four by using conjugation identities:
 \begin{align*}
 \alpha &=\chi(a)=\chi(ab), \\ 
 \bar\alpha &=\chi(a^{-1})=\chi( b^{-1} a^{-1}), \\
 \beta &=\chi(b)=\chi(  a^{-1}  b), \\ 
 \bar\beta &=\chi(b^{-1})=\chi(a b^{-1}). 
 \end{align*}
%The next lemma give the final result: 
 
 \begin{lemma}\label{lem:projection}
 The projection
$\pi(\res (X(\Gamma)))$
has three components:
\begin{align*}
 U_0 &= \{ (\alpha,\bar\alpha,\beta,\bar \beta)\in\mathbb{C}^4\mid \alpha=\bar\alpha, \beta=\bar\beta \}, \\
 U_1 &= \{ (\alpha,\bar\alpha,\beta,\bar \beta)\in\mathbb{C}^4\mid \alpha=\bar\alpha=1 \}, \\
 U_2 &= \{ (\alpha,\bar\alpha,\beta,\bar \beta)\in\mathbb{C}^4\mid \beta=\bar\beta=1 \}. 
\end{align*}
\end{lemma}

\begin{proof}
We mimic the proof for $\SL(2,\CC)$ in Section~\ref{section:SL2}, i.e.\
we look at the fixed points in $\CC^8$ of the action of the monodromy (conjugation by $t$).
Since for each character there is a unique semisimple representation, there is always a matrix that realizes the conjugation,
and this can be taken as the image of $t$ for extending the character to $\Gamma$.  
Four of the identities that  we get by looking at the fixed points are equivalent to the reduction from eight to four coordinates.
We discuss the other four identities.
 
The first identity comes from taking traces on the image of
$$
t  a  b t^{-1}=a  b^2a b .
$$
To express the trace of $a  b^2a b $ using Lawton's coordinates we use the $\SL(3,\CC)$ analog of the identities for $\SL(2,\CC)$ in \eqref{eqn:basicSL2}.
Denote the images in $\SL(3,\CC)$ by capital letters.
From the characteristic polynomial identity we have
$$
(B A)^3-\alpha (B A)^2+ \bar \alpha   B A-\mathrm{Id}=0.
$$
Multiplying by $A^{-1} $ we get
$$
BABAB=\alpha BAB -\bar \alpha B+  A^{-1}.
$$ 
Thus, by taking traces we get the equation 
\begin{equation} \label{eqn:a1}
\alpha=\alpha\beta-\bar\alpha\beta+\bar\alpha.
\end{equation}
Similarly, for $ t  (a b)^{-1}t^{-1}=(a  b^2a b )^{-1}$
we get
\begin{equation} \label{eqn:a2}
\bar\alpha=\bar\alpha\bar\beta-\alpha\bar\beta+\alpha . %\bar\alpha.
\end{equation}
The  third identity comes from $t  b t^{-1}=b a b$. 
From the characteristic polynomial identity we have
$$
B^3-\beta B^2+ \bar \beta  B -\mathrm{Id}=0.
$$
Multiplying by $B^{-1}A$ we get
$$
B^2  A =\beta BA -\bar \beta A+ B^{-1}  A.
$$
Thus, we get the equation 
\begin{equation} \label{eqn:a3}
 \beta= \beta \alpha -\bar\beta \alpha + \bar\beta.
\end{equation}
Similarly, from $t  b^{-1} t^{-1}=(b a b)^{-1}$ we get
\begin{equation} \label{eqn:a4}
 \bar\beta= \bar\beta \bar\alpha -\beta \bar\alpha + \beta.
\end{equation}
The statement follows from identities \eqref{eqn:a1}--\eqref{eqn:a4}. 
Notice that we do not need to compute more identities because of \cite{Lawton0}, 
and because $t[a,b]t^{-1}=[a,b]$.
\end{proof}

To get all the ambient coordinates we need a new variable: 
$$
\eta(\chi)=\chi([a,b])\, .
$$We know by \cite{Lawton0}
that 
\begin{equation}
 \label{eqn:polcommutator}
\eta^2-P \eta + Q=0,
 \end{equation}
for some polynomials $P,Q\in\ZZ[\alpha,\beta,\bar \alpha,\bar\beta]$. 
Using Lemma~\ref{lem:projection} and  by replacing the values of $P$ and $Q$ in \cite{Lawton0},
we obtain:

\begin{lemma} \label{lem:Wj}
 $W=\res (X(\Gamma,G))$ has three components $W_0$, $W_1$ and $W_2$, 
each $W_i$ being a two-fold ramified covering of $U_i$
 according to \eqref{eqn:polcommutator}.
 \begin{itemize}
  \item For $W_0$ the polynomials and the discriminant are
\begin{align*}
 Q =& \,  \alpha^4 \beta^2+ \alpha^2 \beta^4-2  \alpha^4\beta-4
  \alpha^3 \beta^2-4  \alpha^2 \beta^3-2 \alpha  \beta^4 +  \alpha^4\\
 & +2  \alpha^3\beta+12  \alpha^2 \beta^2+2
 \alpha  \beta^3+ \beta^4+4  \alpha^3+4  \beta^3-12  \alpha^2-12  \beta^2+9,\\[1ex]
 P=& \,  \alpha^2 \beta^2-2  \alpha^2\beta-2 \alpha  \beta^2+2 \alpha^2+2  \beta^2-3, \\[1ex]
P^2-4Q=& \, (\alpha^2 \beta^2-6 \alpha\beta -4 \alpha-4 \beta-3)(\alpha\beta-2\alpha-2\beta+3)^2.
\end{align*}
  \item For $W_1$, we have  
\begin{align*}
 Q=& \,  \bar \beta^3+ \beta^3-3  \beta\bar \beta+2, \\[1ex]
P=& \,  \beta\bar \beta-\beta-\bar \beta-1, \\[1ex]
P^2-4Q=& \,  \beta^2 \bar \beta^2-4  \beta^3-4   \bar \beta  ^3-2  \beta^2\bar \beta-2 \beta  
\bar \beta  ^2+{\beta
}^2+  \bar \beta  ^2+12 \beta
\bar \beta+2 \beta+2 \bar \beta -7.
\end{align*}
\item For $W_2$, we have
\begin{align*}
 Q=& \, \alpha^3+  \bar \alpha^3-3 \alpha\bar \alpha+2 , \\[1ex]
P=& \, \alpha\bar \alpha-\alpha-\bar \alpha-1,\\[1ex]
P^2-4Q=& \,\alpha^2 \bar \alpha^2-4  \alpha^3-4   \bar \alpha  ^3-2  \alpha^2
\bar \alpha-2\alpha\bar\alpha^2+ \alpha^2+ \bar\alpha^2+12 \alpha\bar \alpha+2 \alpha+2 \bar \alpha-7.
\end{align*}
 \end{itemize}
\end{lemma}

We describe the image of the set of reducible characters:

\begin{lemma}
\label{lemma:intersect}
The image $\res (X_{red}(\Gamma,G))\subset W$ is the curve 
 $$
 \{(\alpha,\beta,\eta)\in W_0\mid \alpha\beta -2 \alpha-2 \beta+3=0\}.
 $$
All characters in this curve are of type $(2,1)$, except at the point $\alpha=\beta=3$, that corresponds to totally reducible representations. 

In particular $W_1$ and $W_2$ only meet restrictions of reducible representations at their common intersection,
which is the intersection with $W_0$, the point $\alpha=\bar \alpha=\beta=\bar\beta=1$, $\eta=-1$.
\end{lemma}

\begin{proof}
 We know from Section \ref{sec:reducible} that if $x_1$ and $x_2$ denote the traces 
 in $\SL(2,\CC)$ of the images of $a$ and $b$, then the 
 image of  
 $X(\Gamma, \SL(2,\CC))$ in $X(F_2, \SL(2,\CC))$ is given by the equation
 $x_1 x_2-x_1-x_2=0$. Now the lemma follows from the identities
 $\alpha=\bar\alpha=x_1+1$, $\beta=\bar\beta=x_2+1$.
\end{proof}

\begin{remark}
\label{rem:disc}
 The discriminant locus of $W_0$ contains $\alpha \beta -2 \alpha-2 \beta+3=0$.
It contains a  second component:
 $$
  \alpha^2 \beta^2-6  \alpha\beta-4 \alpha-4 \beta-3.
 $$
 This is the set of characters of the representations obtained as  composition of representations of $\Gamma$ in $\SL(2,\CC)$
 with $\Sym^2\co \SL(2,\CC)\to \SL(3,\CC)$. This may be seen with the same argument as in 
 Lemma~\ref{lemma:intersect}, but using that
 $\alpha=\bar\alpha= x_1^2-1$ and $\beta=\bar\beta=x_2^2-1$. 
%Note that this curve is rational, being the image of a rational curve.
\end{remark}

Corollary~\ref{cor:fibre} yields:

\begin{corollary}
 \label{cor:cardinality_fibre}
 The fiber of 
 $$
 \res: X_{red}(\Gamma,\SL(3,\CC))\cap \overline{ X_{irr}(\Gamma,\SL(3,\CC))}\to \{ (\alpha,\beta,\eta)\in W_0\mid \alpha\beta -2 \alpha-2 \beta+3=0\}
 $$
 has cardinality $6$, except for 
 $(\alpha,\beta)= (1,1)$, $(3,3)$, or $\left(\frac{1\pm\sqrt{5}}{2}, \frac{1\mp\sqrt{5}}{2}\right)$,
 where it has cardinality $3$.
\end{corollary}

\begin{remark}
%  For  $\alpha=\bar \alpha=\beta=\bar\beta=1$, we have a curve of partially reducible representations, 
% with equations $y^3 - y^2 \bar{y}^2 + \bar{y}^3=0$. 
% Equivalently, $w=v^2$, which gives the parametrization
% $y = v + 1/v^2, \bar{y} = v^2 + 1/v$, $v \in \CC^*$ ($v=\xi=\phi(t)$ in the proof of Lemma~\ref{lemm:redtoirreduc}).
%
The intersection of the two components of the discriminant (see Remark~\ref{rem:disc}) is the pair of points
 $(\alpha,\beta)= \{(3,3), \left(\frac{1\pm\sqrt{5}}{2}, \frac{1\mp\sqrt{5}}{2}\right)\}$. 
 The third point $(\alpha,\beta)= (1,1)$ in Corollary~\ref{cor:cardinality_fibre} yields singular points in $V_0$.
\end{remark}

A character of $\Gamma$ that is irreducible may have a restriction to $F_2$ that is reducible.

\begin{lemma}
 \label{lem:redired}
The characters of $\res(X_{irr}(\Gamma, G))$  that are $F_2$-reducible are 
characters of representations $\rho$ whose
restriction to $F_2$ is totally reducible, and $\rho(t)$ acts as a cyclic 
permutation of the invariant subspaces of 
$\rho(F_2)$. In particular $\rho$ is metabelian and  $\tr(\rho(t^{\pm 1}))=0$.
 \end{lemma}
 
 \begin{proof}
 Let $\rho\in R(\Gamma,G)$ be an irreducible representation whose restriction
 to $F_2$ is reducible.
If $\rho\vert_{F_2}$ was partially reducible then, as the invariant subspaces
would have different dimension, they should also be invariant by $\rho(t)$, and therefore
$\rho$ would be reducible. 
Thus the only possibility is that  $\rho\vert_{F_2}$ is totally reducible and $\rho(t)$ cyclically permutes its three invariant subspaces.
\end{proof}

Irreducible metabelian characters
in $\SL(3,\CC)$
have been 
studied by Boden and Friedl in   \cite{BF2, BF3},
who prove that in 
$X(\Gamma,\SL(3,\CC))$ there are five of them:

\begin{corollary}
 \label{cor:redired}
 There are five characters in $\res(X_{irr}(\Gamma, G))$ that are $F_2$-reducible:
 one in
$W_0$, $\alpha=\beta=-1$; two in $W_1$, $(\beta,\bar\beta)=(-1\pm 2 i,-1\mp 2 i)$; and
two in $W_2$, $(\alpha,\bar\alpha)=(-1\pm 2 i,-1\mp 2 i)$.  In all cases $\eta=3$.
  \end{corollary}

Those yield precisely five characters in $X(\Gamma, G)$ for $G=\SL(3,\CC)$ and   $G=\PSL(3,\CC)$ 
(they satisfy $y=\bar y=z=\bar z=0$) and five subvarieties for $G= \GL(3,\CC)$. 
To study further the fibre of $\res\co X_{irr}(\Gamma,\SL(3,\CC))\to W$, notice that the fibre
of $F_2$-irreducible characters in $W$ has precisely three points.
Namely, if $\rho\in R(\Gamma, \SL(3,\CC))$
satisfies that $\rho\vert_{ F_2}$ is irreducible, the relations in 
\eqref{eq:pres-fiber} yield that $\rho(t)$ is unique up to the action of $\mu_3$, 
the center of $\SL(3,\CC)$. 
It is straightforward to check that 
those three choices yield the same character if and only if
$\rho$ is a metabelian irreducible representation as in Lemma~\ref{lem:redired}.
Thus the five characters in Corollary~\ref{cor:redired} are the ramification points of 
$\res\co X_{irr}(\Gamma,\SL(3,\CC))\to W$.
Notice that those five characters also lie in the
zero locus of the discriminant $P^2-4 Q=0$, as they satisfy $\eta=3$.
In \cite{BF3} Boden and Friedl prove that these representations are smooth points of $X(\Gamma,\SL(3,\CC))$, 
so  $\overline{X_{irr}(\Gamma,\SL(3,\CC))}$ has the same number of components as $W$. 
Summarizing, we have:

\begin{proposition}
 \label{prop:3components}
The set $\overline{X_{irr}(\Gamma,\SL(3,\CC))}$ has  three components $V_0$, $V_1$ and $V_2$, that are respective
3 to 1 branched covers
 of $W_0$, $W_1$ and $W_2$. 
 
 The branching points in  $X_{irr}(\Gamma,\SL(3,\CC))$ are the five metabelian irreducible characters in Corollary~\ref{cor:redired}.
\end{proposition}

If we add the variables
 $$
 y(\chi)= \chi(t), \quad \bar{y}(\chi)=\chi(t^{-1}),
 $$
then the variables $(\alpha,\bar\alpha,\beta,\bar\beta, \eta, y,\bar{y})$ 
describe $X_{irr}(\Gamma,\SL(3, \CC))$, and the map 
 $$
(\alpha,\bar\alpha,\beta,\bar\beta, \eta, y,\bar{y}) \mapsto (\alpha,\bar\alpha,\beta,\bar\beta, \eta)
 $$ 
is three-to-one except at $(y,\bar{y})=(0,0)$. 
Actually, it is the quotient by 
the action of the center $(y,\bar{y})\mapsto (\varpi y,\varpi^2 \bar{y})$, 
where $\varpi\in \mu_3$.
Notice that in Lemma~\ref{lemm:redtoirreduc} we have shown that reducible characters 
in $\overline{X_{irr}(\Gamma,\SL(3,\CC))}$  
are also determined by the values of $y$ and $\bar{y}$. Thus, using Proposition~\ref{prop:parametersLawton}, we get:

\begin{proposition}\label{prop:parameters}
The parameters $(\alpha,\bar\alpha,\beta,\bar\beta, \eta, y,\bar{y})$ 
describe the set $X_{irr}(\Gamma,\SL(3, \CC))$
pointwise (that is, they describe a variety homeomorphic to 
$X_{irr}(\Gamma,\SL(3, \CC))$ in the classical topology). 
The parameters  $(y,\bar y, z,\bar z,\alpha,\bar\alpha,\beta,\bar\beta)$, with
$z=\chi(t a^{-1} t a)$ and $\bar z=\chi(a^{-1} t^{-1} a t^{-1})$, describe 
$X_{irr}(\Gamma,\SL(3, \CC))$ scheme-theoretically.
\end{proposition}

Here we say \emph{scheme-theoretically} because the natural defining ideal is perhaps non-radical.
The Zariski tangent space at a character $\chi$ is denoted by   $T^{Zar}_\chi X(\Gamma, \SL(3,\CC))$,
without assuming that it is scheme-reduced.

\begin{lemma}
 \label{lemma:h1inj}
 Let $\chi\in X(\Gamma, \SL(3,\CC))$ be a character such that $\res(\chi)$ is irreducible. The restriction map induces an isomorphism
 $$
  T^{ Zar}_\chi X(\Gamma, \SL(3,\CC)) \cong  T^{ Zar}_{\res(\chi)} \big( X(F_ 2,\SL(3,\CC)) ^{\phi^*} \big),
 $$
 where $\phi\co F_2\to F_2$ denotes the action of the monodromy (conjugation by $t$), $\phi^*$ the induced map, and 
 $X(F_ 2,\SL(3,\CC)) ^{\phi^*} $ the fixed point set.
\end{lemma}

\begin{proof}
We have a natural isomorphism:
$$
 T^{ Zar}_{\res(\chi)} \big( X(F_ 2,\SL(3,\CC)) ^{\phi^*} \big)\cong  \big( T^{ Zar}_{\res(\chi)} X(F_ 2,\SL(3,\CC)) \big)^{\phi^*},
$$
where $\phi^*$ still denotes the tangent map. 
By Weil's theorem, those Zariski tangent spaces (as schemes) are naturally isomorphic to  cohomology groups:
$$
T^{ Zar}_\chi X(\Gamma, \SL(3,\CC))\cong  H^1(\Gamma,\mathfrak{sl}(3,\CC)_{\operatorname{Ad}\rho})
\quad\textrm{ and }\quad
 T^{ Zar}_{\res(\chi)} X(F_ 2,\SL(3,\CC)) \cong  H^1(F_2,\mathfrak{sl}(3,\CC)_{\operatorname{Ad}\rho})
$$
for any representation $\rho$ with character $\chi$ (cf.~\cite{LM}), where $ \operatorname{Ad}\rho$
denotes the adjoint representation. We let $\phi^*$ denote also the induced map in cohomology,
as it corresponds to the tangent map. The Lyndon-Hochschild-Serre spectral sequence applied  to 
$
1\to F_2\to\Gamma\to\ZZ=\langle t\rangle\to 1
$
yields:
$$
 H^1(\ZZ,\mathfrak{sl}(3,\CC)^{F_ 2}_{\operatorname{Ad}\rho})\to
  H^1(\Gamma,\mathfrak{sl}(3,\CC)_{\operatorname{Ad}\rho}) \to
  H^1(F_2,\mathfrak{sl}(3,\CC)_{\operatorname{Ad}\rho})^{\phi^*} \to
  H^2(\ZZ,\mathfrak{sl}(3,\CC)^{F_ 2}_{\operatorname{Ad}\rho}),
$$
cf.~\cite[6.8.3]{Weibel}. Thus it suffices to show that the invariant subspace  
$\mathfrak{sl}(3,\CC)^{F_ 2}_{\operatorname{Ad}\rho}$ is trivial.
By contradiction, assume that  there exists a nonzero $\theta\in \mathfrak{sl}(3,\CC)$ $F_ 2$-invariant,
namely ${\operatorname{Ad}}_{\rho( a)}(\theta)= {\operatorname{Ad}}_{\rho( b)}(\theta)=\theta$. 
This implies that
$\exp(\lambda \theta)$ commutes with $\rho(a)$ and $\rho(b)$, for every $ \lambda \in\mathbb R$,
 hence the restriction $\rho\vert_{F_ 2}$ 
is reducible.  
\end{proof}

\begin{remark}
\label{rem:smoothW}
The singular locus \emph{as schemes} of the components of $W$ is 
the set of $F_2$-reducible characters of $W$, as this is the singular
locus of the discriminant in Lemma~\ref{lem:Wj}. Hence it is the union of the five ramification points 
in Corollary~\ref{cor:redired} and the curve $\res(X_{red}(\Gamma,\SL(3,\CC))$ in 
Lemma~\ref{lemma:intersect}.
\end{remark}

In the previous remark we say \emph{as schemes} because 
during the computation of $U_i$ and $W_i$ we just use traces on group relations
and Lawton's theorem, and we never compute radicals of the ideals

\begin{proposition}
\label{prop:smooth}
 The components $V_0$, $V_1$ and $V_2$ are smooth (and scheme reduced) everywhere except possibly at
  $(\alpha,\beta) = (1,1)$.
 \end{proposition}

Scheme reduced at one point means that the local ring is reduced. This property, together with smoothness,
holds when the
dimension of the Zariski-tangent space equals the dimension of the irreducible component. 
 
\begin{proof}
Let $\chi \in V_0\cup V_1\cup V_2$ be an irreducible character. If $\chi $ is one of the five ramification points 
of Corollary~\ref{cor:redired},
 then it is smooth and scheme reduced 
by \cite{BF3}.
Otherwise, $\chi\vert_{F_2}$ is irreducible,  thus Lemma~\ref{lemma:h1inj} applies and we get that 
$ T^{ Zar}_\chi X(\Gamma, \SL(3,\CC))$ is isomorphic to $T^{ Zar}_{\res(\chi)} \big( X(F_ 2,\SL(3,\CC)) ^{\phi^*} \big)$.
In particular, if $\chi\not\in V_0\cap (V_1 \cup V_2)$,
then
$ T^{ Zar}_\chi V_0$ is isomorphic to $T^{ Zar}_{\res(\chi)} W_0$ and by Remark~\ref{rem:smoothW}
it is a two dimensional space.

For irreducible characters in $V_0\cap V_1$ or in $V_0\cap V_2$, the cohomology groups yield
the tangent spaces to $V_0\cup V_1\cup V_2$ and $W=W_0\cup W_1\cup W_2$, thus we need to add 
linear conditions on the ambient coordinates to distinguish components ($\alpha=\bar\alpha$, $\beta=\bar\beta$, 
$\alpha=1$, or $\beta=1$). From those linear conditions we easily get
the dimension of the Zariski tangent space.

Reducible characters are discussed in Lemma~\ref{lemm:redtoirreduc}, that yields smoothness for each reducible character in $V_0$ 
except for
$\alpha=\beta=1$ and $\alpha=\beta= 3$. Smoothness for this last point can be checked using the arguments of 
\cite[Thm.~1.3]{HPAGT}.
\end{proof}

%%%%%%%%%%%%%%%%%%%%%%%%%%%%%%%%%%%%%%%%%%
\section{Description of the non-distinguished components}\label{sec:nondist}
%%%%%%%%%%%%%%%%%%%%%%%%%%%%%%%%%%%%%%%%%%

Now we move to get equations for $V_0,V_1,V_2$. We name $V_0$ the \emph{distinguished component}
of the character variety $X(\Gamma,\SL(3,\CC))$ since it is the one containing representations
coming from $\Sym^2: \SL(2,\CC)\to \SL(3,\CC)$. Accordingly, $V_1,V_2$ will be called 
non-distinguished components.

In this section we find explicit equations of the non-distinguished components $V_1,V_2$. 
We start with $V_2$. We keep on working with Presentation~\eqref{eq:pres-fiber} and using capitals to denote the images of 
$a$, $b $ and $t\in \Gamma$ in $\SL(3,\CC)$.

\begin{proposition}\label{prop:eta-non-distinguished-V2}
 The non-distinguished component $V_2$ is described as follows. Taking coordinates
 $\alpha=\tr(A)$, $\bar\alpha=\tr(A^{-1})$,  $\beta=\tr(B)$, $\bar\beta=\tr(B^{-1})$, 
 $\eta=\tr([A,B])$, 
  $y=\tr(T)$,  $\bar{y}=\tr(T^{-1})$, $z=\tr(TA^{-1} TA)$, $\bar z=\tr(A^{-1} T^{-1} AT^{-1})$, 
  the equations satisfied by $V_2$ are
  $\beta=\bar\beta=1$, $y  \bar{y} = \alpha+\bar\alpha+2$,
 $y^3 + \bar{y}^3 =\alpha\bar\alpha+ 5 \alpha+5\bar\alpha+ 5$. 
$\eta  =\bar{y}^3-3(\alpha+\bar\alpha+1)$, $z = y^2 - \bar{y}$, $\bar{z} = \bar{y}^2 -y$.
 
 The only reducible
 representations are $\mu_3\cdot (2,2,2,2,1,1,1,1)$ and are partially reducible.
 % with $\alpha=\bar\alpha=1$, $(y,\bar{y})=
%(2,2)$, $(2\varpi,2\varpi^2)$, $(2\varpi^2,2\varpi)$, $\varpi=\vvarpi$. 
\end{proposition}

\begin{proof}
By Lemma \ref{lem:Wj}, $W_2$ is described as the double cover of the plane 
$(\alpha,\bar\alpha)$, ramified over the curve  $\Delta=0$, where
 $$
 \Delta=   \alpha^2 \bar \alpha^2-4  \alpha^3-4   \bar \alpha  ^3-2  \alpha^2
\bar \alpha-2   \alpha\bar \alpha^2+ \alpha^2+  
\bar \alpha^2+12 \alpha\bar \alpha+2 \alpha+2 \bar \alpha-7.
 $$
Thus the ring of functions of $W_2$ is $\QQ[\alpha,\bar\alpha][\sqrt{\Delta}]$.
We have $\beta=\bar\beta=1$, so the matrix $B$ has eigenvalues $1,i,-i$. We fix a basis that
diagonalizes $B$, so that we can write
 $$
 B=\left(
\begin{array}{ccc}
 1 & 0 & 0 \\
 0 & i & 0 \\
 0 & 0 & -i 
\end{array}
\right).
 $$
Assume that the matrix $A$ has non-zero entries $(2,1)$ and $(3,1)$. Rescaling the basis vectors, we can write
 $$
A=\left(
\begin{array}{ccc}
 a & b & c \\
 1 & d & e \\
 1 & f & g 
\end{array}
\right).
$$
Solving the equations $\tr(A)=\tr(AB)$, $\tr(B)=\tr(BA^{-1})$, $\tr(B)=\tr(B^2A)$ and the
equations for the inverses  $\tr(A^{-1})=\tr(A^{-1}B^{-1})$, $\tr(B^{-1})=\tr(B^{-1}A)$, $\tr(B^{-1})=\tr(B^{-2}A^{-1})$,
together with $\det(A)=1$, $\alpha=\tr (A)$ and $\bar\alpha=\tr(A^{-1})$, we get the coefficients
\cite{fig8html}
 $$
A=\left(
\begin{array}{ccc}
 \frac{\alpha+1}{2} &  \frac{1-i}{8}(\alpha^2-2 \bar\alpha+1) &\frac{1+i}{8} (\alpha^2-2 \bar\alpha+1 ) \\
 1 & \frac{1-i}{4}(\alpha-1) & \frac{1+i}{4}\frac{(\alpha^3-\alpha^2-4\alpha \bar\alpha-4 \alpha+5+2i\sqrt\Delta)}{\alpha^2-2 \bar\alpha+1} \\
 1 &  \frac{1- i}{4}\frac{(\alpha^3-\alpha^2-4\alpha \bar\alpha-4 \alpha+5-2i\sqrt\Delta)}{\alpha^2-2 \bar\alpha+1} 
& \frac{1+i}{4}(\alpha-1) 
\end{array}
\right).
$$
This matrix is well-defined off the set $\alpha^2-2\bar\alpha+1=0$.

Now, solving the equations $TA=ABT$ and $TB=BABT$, we get a one-dimensional space of matrices $T$ spanned by
{\tiny $$
T_0= \left(
\begin{array}{ccc}
   \begin{minipage}{5cm}{$ (\alpha^2-2 \bar\alpha+1) \big(((1+i) \alpha+3-i)\sqrt{\Delta} 
 -(2-2i) \bar\alpha^2-(6-6i) \alpha \bar\alpha+(2-2i) \bar\alpha-(1+i) \alpha^3
+(-1+3i) \alpha^2+(3-i)\alpha+(7-3i)\big) $ } \end{minipage}
   & \begin{minipage}{5cm}{$\frac{1}{4} (\alpha^2-2 \bar\alpha+1)^2 \big((-1-i)\sqrt{\Delta} +
2i \alpha^2+(1-i)\alpha+(1-i)\alpha\bar\alpha+(5-5i)\bar\alpha-(5-5i)\big) $ } \end{minipage}
   & \begin{minipage}{5cm}{$\frac{1}{4} (\alpha^2-2 \bar\alpha+1) \big(
   (-(1+i) \alpha^2-(4-4i) \alpha- (2+2 i) \bar\alpha +(7-i))\sqrt{\Delta}
   +   (6-6i)\alpha\bar\alpha^2+(-2-6 i)\bar\alpha^2 
   -(1-i)\alpha^3\bar\alpha+(3+5 i)\alpha^2\bar\alpha+(1+15 i)\alpha\bar\alpha 
   -(3-11i) \bar\alpha + 2\alpha^4+(-1-7i)\alpha^3+(-5-5i) \alpha^2
   +(11+13 i)\alpha +(-3-19 i)   \big)$ } \end{minipage} \\[7pt]
\begin{minipage}{5cm}{$-2(\alpha^2-4 i\alpha+2 \bar\alpha-(3-4 i))\sqrt{\Delta} 
-12i \alpha\bar\alpha^2-(8+4i) \bar\alpha^2+ 2i\alpha^3\bar\alpha+(-8+2i)\alpha^2\bar\alpha+
(16+14i)\alpha \bar\alpha +(8+14i) \bar\alpha +
(2-2i) \alpha^4 +(-8-6i) \alpha^3 +(-4+6i)\alpha^2+(1-12 i)\alpha-(8-11 i)$ } \end{minipage}
    &  \begin{minipage}{5cm}{$ (\frac{1}{2}(\alpha^2-2 \bar\alpha+1) \big((2 \alpha+2-4 i)\sqrt{\Delta} +4i \bar\alpha^2+(8+ 4i) \alpha \bar\alpha-4i \bar\alpha
-2 \alpha^3 +(2+4 i) \alpha^2 +(2-4 i)\alpha-(10+4i )\big) $ } \end{minipage}
   & \begin{minipage}{5cm}{$4(\bar\alpha   \alpha+\alpha+\bar\alpha-3)\sqrt{\Delta}  
   +i \alpha^5+10i \alpha^3-15i \alpha+4i \bar\alpha^3+
   4i \alpha^2\bar\alpha^2+12i\alpha\bar\alpha^2+4i\bar\alpha^2 -
   i\alpha^4\bar\alpha+4i\alpha^3\bar\alpha-6i\alpha^2\bar\alpha+36i\alpha\bar\alpha 
   +17i \bar\alpha+32$ } \end{minipage} \\[7pt]
\begin{minipage}{5cm}{$  (\alpha^2-2 \bar\alpha+1) \big(-2i\sqrt{\Delta} +(-2+2i) \alpha^2+2\alpha+2\alpha \bar\alpha+(6-4 i) \bar\alpha-(8-2i) \big) $ } \end{minipage}
   & \begin{minipage}{5cm}{$(\alpha^2-2 \bar\alpha+1)^2 (\alpha-\bar\alpha-4 i) $ } \end{minipage}
   &\begin{minipage}{5cm}{$\frac{1}{2} (\alpha^2-2\bar\alpha+1) \big((2i \alpha+4+2i)\sqrt{\Delta}
   -2i\alpha^3+(-4+2 i) \alpha^2+ (-4+8 i)\alpha \bar\alpha+(4+2 i)\alpha
   +2i \bar\alpha^2- 2i\bar\alpha-(5+2i) \big) $ } \end{minipage}
\end{array}
\right)
 $$}
Let $y_0=\tr (T_0)$, $\bar{y}_0=\det(T_0) \tr( T_0^{-1})$, $d_0=\det (T_0)$. 
It is readily computed that
{\small
 \begin{align*}
 y_0 =& \, 2 (\alpha^2-2 \bar\alpha+1) \Big( \big((1+i) \alpha+3-i\big)\sqrt{\Delta} -(1+i) \alpha^3 +(-1+3 i) \alpha^2+ (2+6 i) \alpha \bar\alpha+(3- i)\alpha  \\
   & +(-2+2i) \bar\alpha^2+(2-2i) \bar\alpha-(7-7i) \Big), 
 \\ %\end{align*} \begin{align*}
 \bar{y}_0 =& \, 4 (\alpha^2-2 \bar\alpha+1)^2  \Big( \big(-i \alpha^4-4 \alpha^3+ 6 i \alpha^2+(4i-4)  \alpha^2\bar\alpha+12 \alpha +(2+4 i)4 \alpha  \bar\alpha +4 i\bar\alpha^2+(4i-4 ) \bar\alpha\\
   &-(33i+8 )\big) \sqrt{\Delta}+i\alpha^5 \bar\alpha-3i\alpha^5+ (4-3i)\alpha^4 \bar\alpha-(12+11i\alpha^4 +(4-4 i) \alpha^3 \bar\alpha^2+(-8+6 i) \alpha^3 \bar\alpha\\
   &+(-28+18 i)\alpha^3 -(20+4 i) \alpha^2\bar\alpha^2+(16+46i)\alpha^2 \bar\alpha 
   +(20+2i)\alpha^2 -(16+4 i)  \alpha \bar\alpha^3-(20+8 i) \alpha  \bar\alpha^2\\
   &+(120-35 i)  \alpha \bar\alpha+(12-35 i) \alpha  
   +(-16+12i)\bar\alpha^3+ (4-8i)\bar\alpha^2+(28+9i)\bar\alpha -(88 -21i) \Big), 
 \\ %\end{align*} \begin{align*}
 d_0 =&\, 8(1-i) (\alpha^2-2 \bar\alpha+1)^3 \Big( \big(-\alpha^6+(6+6 i) \alpha^5+ (4+6 i) \alpha^4\bar\alpha+(17-36 i)\alpha^4+(8-24 i)\alpha^3  \bar\alpha \\
  & -(84+20 i)\alpha^3 +  (12-24 i) \alpha^2\bar\alpha^2-(72-60 i)\alpha^2 \bar\alpha+(21+120 i)\alpha^2- (72+72i) \alpha \bar\alpha^2 \\
  &+(136+104i)\alpha\bar\alpha+(134 -122i)\alpha  -8 i \bar\alpha^3-(36-24 i) \bar\alpha^2+(116+22 i) \bar\alpha-(189+36 i)\big)\sqrt{\Delta} \\
  &+ \alpha^7 \bar\alpha-7 \alpha^7  +(3+42 i)\alpha^6 -(1+6 i)\alpha^6  \bar\alpha+(-8-6 i)\alpha^5  \bar\alpha^2+(37+24 i)\alpha^5  \bar\alpha+(109-42 i)\alpha^5 \\
  &+ (-32+66 i) \alpha^4\bar\alpha^2+(107-190 i)\alpha^4 \bar\alpha-(201+152 i)\alpha^4+ (4+24 i)\alpha^3 \bar\alpha^3+(76-12 i) \alpha^3\bar\alpha^2 \\
  &-(369+104 i) \alpha^3\bar\alpha-(97-396 i)\alpha^3+(156-96 i)\alpha^2  \bar\alpha^3-(236+4 i) \alpha^2 \bar\alpha^2+(337+590 i)\alpha^2  \bar\alpha \\
  &+(373-46 i)\alpha^2 +(48+8 i)  \alpha \bar\alpha^4-(132+200 i)\alpha  \bar\alpha^3-(260-530 i) \alpha \bar\alpha^2 +(395-704 i)\alpha  \bar\alpha\\
  &-(309+426 i)\alpha +(48-56 i) \bar\alpha^4-(92-144 i) \bar\alpha^3+(76+98 i) \bar\alpha^2 
  +(69-378 i) \bar\alpha-(95-500 i)\Big).
\end{align*}
}

The choices $T\in \SL(3,\CC)$ are of the form $\varpi d_0^{-1/3}\, T_0$, $\varpi\in \mu_3$. 
Therefore, the
coordinates  $y= \tr (T)$, $\bar{y}=\tr (T^{-1})$ are given by $y=y_0\,d_0^{-1/3}$, 
$\bar{y}= \bar{y}_0\, d_0^{-2/3}$.
A computer calculation gives
 \begin{align*} %\label{eqn:y,z}
 y  \bar{y} &= \alpha+\bar\alpha+2, \nonumber\\
y^3 &=\frac12 (\alpha\bar\alpha + 5 \alpha+5\bar\alpha+5 ) +\frac12 \sqrt{\Delta} \, ,\\
\bar{y}^3 &=\frac12 (\alpha\bar\alpha + 5 \alpha+5\bar\alpha+ 5) -\frac12 \sqrt{\Delta} \, , \\
  \eta &=\tr ([A,B])= \frac12 (\alpha\bar\alpha-\alpha-\bar\alpha -1) -\frac12 \sqrt\Delta\, .
 \end{align*}
Note that we only compute quantities that are $\mu_3$-invariant. 
First we note that $\eta  =\bar y^3-3(\alpha+\bar\alpha+1)$.
%, so $\eta$ is unnecessary to describe $V_2$.

The above equations describe a Zariski open set defined by 
$d_0\neq 0$, $\alpha^2-2\bar\alpha+1 \neq 0$. 
When we approach a point on  the set $d_0=0$, the matrix $T_0$ above becomes singular. But the
normalized matrix $T=d_0^{-1/3} T_0$ has values $y=\tr(T), \bar{y}=\tr(T^{-1})$ that tend to
finite well-defined numbers. Moreover, the representations on $X(\Gamma,F_2)$ given
by $(A,B)$ and $(BA,BAB)$ are conjugated since they have the same Lawton's coordinates.
And as they are irreducible, the conjugation matrix $T$ is uniquely defined (up to the
action of $\mu_3$). This gives well-defined $(y,\bar{y})$ up to $\mu_3$, and by continuity, they
satisfy the equations above.

For dealing with the set $\alpha^2-2\bar\alpha+1 =0$, we can argue by continuity as above. 
Else we can parametrize the set of matrices $A$ in which either the entry $(2,1)$ is zero or
the entry $(3,1)$ is zero. This gives the matrices:
$$
\left(
\begin{array}{ccc}
 \frac{\alpha+1}{2} & 0 & \frac{1}{8} \left(-\alpha^3-\alpha^2-3 \alpha+5\right) \\
 1 & \left(\frac{1}{4}-\frac{i}{4}\right) (\alpha-1) & \frac{1}{8} \left(-\alpha^2-2 \alpha-5\right) \\
 0 & 1 & \left(\frac{1}{4}+\frac{i}{4}\right) (\alpha-1) \\
\end{array}
\right), 
\left(
\begin{array}{ccc}
 \frac{\alpha+1}{2} & \alpha-1 & 0 \\
 0 & \left(\frac{1}{4}-\frac{i}{4}\right) (\alpha-1) & 1 \\
 1 & \frac{1}{8} \left(-\alpha^2-2 \alpha-5\right) & \left(\frac{1}{4}+\frac{i}{4}\right) (\alpha-1) \\
\end{array}
\right).
$$
There are two matrices because of the double covering over the locus $\alpha^2-2\bar\alpha+1 =0$.

Finally, note that over the point $(\alpha,\bar\alpha)=(1,1)$, we have partially reducible
representations in $X_\mathrm{PR}(\Gamma,\SL(3,\CC))$, and these form a curve, given as (see
Proposition \ref{prop:SL3-reducibles}):
 $$
 1-  (y \bar{y}+3)+(y^3-y \bar{y}+\bar{y}^3+3)-(y \bar{y}-1)^2 =0.
 $$
The intersection of the closure of $W_2^{irr}$ is given by $y\bar{y}=4$, $y^3+\bar{y}^3=16$.
These are the three points $(y,\bar{y})=
(2,2), (2\varpi,2\varpi^2),(2\varpi^2,2\varpi)$, $\varpi=\vvarpi$.

The parameters $(\alpha,\bar\alpha,y,\bar{y})$ describe point-wise the variety $V_2$. By Proposition
\ref{prop:parameters}, we need to add to variables 
$z=\tr(TA^{-1} TA)$ and $\bar z=\tr(A^{-1} T^{-1} AT^{-1})$, 
to describe $V_2$ scheme-theoretically.
An easy computation with the above matrices $A,B,T$ yields that
 \begin{align*}
   z &= y^2 - \bar{y},\\ 
   \bar{z} &= \bar{y}^2 -y. 
 \end{align*}
Note that to describe $V_2$ we only need the variables $(\alpha,\bar{\alpha},y,\bar{y})$
even scheme-theoretically. 
\end{proof}

We can work out the component $V_1$ in a similar way and get the following.

\begin{proposition}\label{prop:eta-non-distinguished-V1}
 The non-distinguished component $V_1$ is described as follows. Taking coordinates
 $\alpha=\tr(A)$, $\bar\alpha=\tr(A^{-1})$,  $\beta=\tr(B)$, 
 $\bar\beta=\tr(B^{-1})$, $\eta=\tr([A,B])$, 
  $y=\tr(T)$,  $\bar{y}=\tr(T^{-1})$, $z=\tr(TA^{-1} TA)$, $\bar z=\tr(A^{-1} T^{-1} AT^{-1})$, 
  the equations satisfied by $V_1$ are $\alpha=\bar\alpha=1$ and  $y  \bar{y} = \beta+\bar\beta+2,
 y^3 + \bar{y}^3 =  \beta\bar\beta+5 \beta+5\bar\beta+ 5$, 
$\eta  ={y}^3-3(\beta+\bar\beta+1)$, $z= \bar y$, $\bar z=y$.
 
 The only reducible
 representations are $\mu_3\cdot (2,2,2,2,1,1,1,1)$ and are partially reducible.
 % with $\alpha=\bar\alpha=1$, $(y,\bar{y})=
%(2,2)$, $(2\varpi,2\varpi^2)$, $(2\varpi^2,2\varpi)$, $\varpi=\vvarpi$. 
\end{proposition}

\begin{remark}
  One can check that the only singular points of $V_2$ are the points in 
  $\mu_3\cdot (2,2,2,2,1,1,1,1)$. We check this as follows: $V_2$ is a six-covering
  of the plane $(\alpha,\bar{\alpha})$, having three preimages over the curve
  $\Delta=0$ and one preimage over the points $(-1\pm 2i, -1\mp 2i)$. The 
  singularities of $\Delta=0$ are at $(-1\pm 2i,-1\mp 2i), (1,1)$, cf. Remark~\ref{rem:smoothW}. 
  We only have to check whether $V_2$ is smooth at those points, which
  can be done by hand.
  
    Using the equations, one can easily see that the points $\mu_3\cdot (2,2,2,2,1,1,1,1)$
  are ordinary double points (locally analytically isomorphic to the surface singularity
  $\{(u_1,u_2,u_3)\in \CC^3 \ | \ u_1^2+u_2^2+u_3^2=0\}$.

  The same happens for $V_2$.
\end{remark}

%%%%%%%%%%%%%%%%%%%%%%%%%%%%%%%%%%%%%%%%%%
\section{Description of the distinguished component} \label{sec:distinguished}
%%%%%%%%%%%%%%%%%%%%%%%%%%%%%%%%%%%%%%%%%%

The distinguished component is the component $V_0$, which is a triple covering of $W_0$. 

\begin{proposition}\label{prop:distinguishedComp}
 The distinguished component is parametrized as follows. Taking coordinates
 $\alpha=\tr(A)$, $\bar\alpha=\tr(A^{-1})$,  $\beta=\tr(B)$, $\bar\beta=\tr(B^{-1})$,
 $\eta=\tr([A,B])$,  $y=\tr(T)$,  $\bar{y}=\tr(T^{-1})$, 
 $z=\tr(TA^{-1} TA)$ and $\bar z=\tr(A^{-1} T^{-1} AT^{-1})$,
 it has equations: 
  \begin{align*}
   \alpha &=\bar\alpha ,\quad \beta=\bar\beta , \\ 
     y  \bar{y} &= (\alpha+1)(\beta+1), \\
       z\bar z&= 2\alpha^2\beta+\alpha^2+1, \\
    y^3 + \bar{y}^3 &= \alpha^2\beta+\alpha\beta^2+ 6 \alpha\beta+3 \alpha+ 3 \beta+2, \\
 z^3 +\bar z^3 &= \alpha^4 \beta^2  + 10 \alpha^2\beta+ 9\alpha^2 - 2 \alpha^3 -2  ,\\
 yz +\bar y \bar z &= \alpha^2\beta+3\alpha\beta+ 3\alpha+1, \\
  \bar{y}^{2} z + y^{2} \bar{z} &=  \alpha^{2} \beta^{2}  + 4
 \alpha^{2} \beta + 2 \alpha^{2} + 4 \alpha \beta + 2 \alpha + 2 \beta + 1,\\
 \bar{y} z^{2} + y\bar{z}^{2} &= \alpha^{3} \beta^{2} +  \alpha^{3} \beta +  4 \alpha^{2} \beta 
 + 3 \alpha^{2} + 5 \alpha \beta + 3\alpha - 1\,.
  \end{align*}  
%  \begin{align*}
%   &  \eta^2 - P \eta + Q =0, \ \text{where $P,Q$ are given in Lemma \ref{lem:Wj},}\\
%   &  y  \bar{y} = (\alpha+1)(\beta+1), \\
%   & y^3 + \bar{y}^3 = \alpha^2\beta+\alpha\beta^2+ 6 \alpha\beta+3 \alpha+ 3 \beta+2, \\
% &z^3 +\bar z^3 = \alpha^4 \beta^2  + 10 \alpha^2\beta+ 9\alpha^2 - 2 \alpha^3 -2  ,\\
% & z\bar z= 2\alpha^2\beta+\alpha^2+1, \\
% &yz +\bar y \bar z = \alpha^2\beta+3\alpha\beta+ 3\alpha+1, \\
%   & (\alpha-\beta) (2\eta-(\alpha^2\beta^2-2\alpha^2\beta-2\beta^2\alpha+2\alpha^2+2\beta^2-3))= \\
%    & \qquad \qquad \qquad \qquad =(\alpha\beta-2\alpha-2\beta+3 )
%(2y^3-(\alpha^2\beta+\alpha\beta^2+ 6 \alpha\beta+3 \alpha+ 3 \beta+2)), \\
% &(1-\alpha)(2\eta-(\alpha^2\beta^2-2\alpha^2\beta-2\beta^2\alpha+2\alpha^2+2\beta^2-3))= \\
% &\qquad\qquad \qquad \qquad  = (\alpha\beta-2\alpha-2\beta+3 ) (2\bar{y}\bar z 
% - (\alpha^2\beta+3\alpha\beta+ 3\alpha+1)), \\
% &
% (\alpha^3\beta+3 \alpha^2-4 \alpha)(2\eta-
% (\alpha^2\beta^2-2\alpha^2\beta-2\beta^2\alpha+2\alpha^2+2\beta^2-3)) = \\
%&\qquad \qquad \qquad \qquad = (\alpha\beta-2\alpha-2\beta+3 ) (2z^3 -
%(\alpha^4 \beta^2  + 10 \alpha^2\beta+ 9\alpha^2 - 2 \alpha^3-2)) .
%  \end{align*}
The intersection with the reducible locus is as follows:
\begin{itemize}
\item $V_0\cap X_\mathrm{TR}$ is given by the three points
$(y,\bar y, z,\bar z,\alpha,\bar\alpha,\beta\bar\beta)=\mu_3 \cdot (4,4,8,8,3,3,3,3)$. These points are smooth points of 
$V_0$ and $X_\mathrm{TR}$ respectively.
\item $V_0\cap X_\mathrm{PR}$ is given by a six-covering of
the curve $\alpha\beta-2 \alpha-2 \beta+3=0$, $\alpha\neq 3$,
ramified over the points $(\alpha,\beta)=(1,1)$ and 
$ \left( \frac{1\pm\sqrt 5}{2},  \frac{1\mp\sqrt 5}{2} \right)$, where there are
only three preimages. For $(\alpha,\beta)=(1,1)$ the preimages are given by $\eta=-1$, 
$(y,\bar y,z,\bar z)=\mu_3\cdot (2,2,2,2)$, and those 
are the same three points as in $V_j \cap X_{\mathrm{PR}}$, $j=1,2$.
For $(\alpha,\beta)= \left( \frac{1\pm\sqrt 5}{2},  \frac{1\mp\sqrt 5}{2} \right)$ the preimages are
  $\eta=3$,  $(y,\bar y,z,\bar z)=\mu_3 \cdot \left(-1,-1, \frac{1\mp\sqrt 5}{2},\frac{1\mp\sqrt 5}{2} 
  \right)$.
\end{itemize}
\end{proposition}

\begin{proof}
By Lemma \ref{lem:Wj}, $W_0$ is described as the double cover of the plane $(\alpha,\beta)$
ramified over the curve  $\Delta'=0$, where 
 $$
 \Delta'= (  \alpha^2 \beta^2-6 \alpha\beta-4 \alpha-4 \beta-3)
   ( \alpha\beta-2 \alpha-2 \beta+3 )^2.
 $$ 
The locus $F=\alpha\beta-2 \alpha-2 \beta+3=0$ corresponds to reducible representations.
Therefore, the ring of functions of $W_0\cap\res(X_{irr}(\Gamma,G))$ is $\QQ[\alpha,\beta][F^{-1},\sqrt{\Delta}]$,
where 
 $$
 \Delta=  \alpha^2 \beta^2-6 \beta \alpha-4 \alpha-4 \beta-3.
 $$
We have $\beta=\bar\beta$, so the matrix $B$ has one eigenvalue equal to $1$. A slice of
the set of such matrices is defined by 
 $$
 B=\left(
\begin{array}{ccc}
 1 & 0 & 0 \\
 0 & \beta-1 & 1 \\
 0 & -1 & 0 \\
\end{array}
\right).
 $$
Assume that the matrix $A$ has non-zero entries $(2,1)$ and $(3,1)$. Rescaling the basis vectors, we can write
 $$
A=\left(
\begin{array}{ccc}
 a & b & c \\
 1 & d & e \\
 1 & f & g 
\end{array}
\right).
$$
Solving the equations $\alpha=\tr (A)=\tr (A^{-1})$, 
$\tr(A)=\tr(AB)$, $\tr(B)=\tr(BA^{-1})$, $\tr (AB^{-1})=\tr(B^{-1})$,
$\tr(A^{-1})=\tr(A^{-1}B^{-1})$ and $\det(A)=1$, we get
\cite{fig8html}
{\small $$
A=\left(
\begin{array}{ccc}
 \hspace{-2mm} \frac{\alpha \beta-2\alpha-\beta}{\beta-3} \hspace{-4mm}& \frac{2 (\alpha \beta-2\alpha-2 \beta+3) (\beta-\alpha)}{(\beta-3)^2 (\beta+1)} &
 \frac{ (\alpha \beta-2\alpha-2 \beta+3) (\beta-\alpha) (\beta-1)}{(\beta-3)^2 (\beta+1)} \\ 
 1 & \frac{4 \alpha^2-\alpha \beta^3+4 \alpha\beta^2-9 \alpha\beta-6 \alpha+7 \beta^2-6 \beta-9
+ (\beta-3) (\beta+1) \sqrt{\Delta}}{2 (\beta-3) (\beta+1) (\beta-\alpha)} \hspace{-3mm} &
\frac{ 4 \alpha \beta^3-  \alpha \beta^4 + 5 \beta^3   - 5 \alpha \beta^2+ 2 \alpha^2 \beta - 8 \beta^2-2 \alpha^2- 2 \alpha \beta  - 9 \beta 
+(\beta-3) (\beta+1)(\beta -2)\sqrt{\Delta}}{2 (\beta-3) (\beta+1) (\beta-\alpha)} \hspace{-2mm}\\
 1 & \frac{4 \alpha^2-\alpha \beta^3+4 \alpha\beta^2-9 \alpha\beta-6 \alpha+7 \beta^2-6 \beta-9
-  (\beta-3) (\beta+1) \sqrt{\Delta}}{2 (\beta-3) (\beta+1) (\beta-\alpha)} \hspace{-4mm}&
 \frac{ \alpha \beta^3 + 2 \beta^3  
- 8 \alpha \beta^2 + 2 \alpha^2 \beta - 5 \beta^2- 2 \alpha^2 + 5 \alpha \beta+ 6 \alpha + 6 \beta+9 
- (\beta-3) (\beta+1)\sqrt{\Delta}}{2(\beta-3) (\beta+1) (\beta-\alpha)} 
\end{array}
\right)
 $$
}
These matrices are well defined for $\beta\neq 3, -1,\alpha$.

Now, solving the equations $TA=ABT$ and $TB=BABT$, we get a one-dimensional space of matrices $T$ spanned by
{\tiny $$
T_0= \left(
\begin{array}{ccc}
   \begin{minipage}{5cm}{$ 
-(\alpha-\beta) (\beta-3)^2 \beta (\beta+1)^2 
 (2 \alpha \beta-\beta^2 - 4\alpha +2\beta-3) \sqrt{\Delta}  +
 (\alpha-\beta) (\beta-3) (\beta+1)^2 (2 \alpha \beta-\beta^2 - 4\alpha +2\beta-3) 
 (3 \beta^2+\beta-\alpha \beta^3+3\alpha  \beta^2-8\alpha-6)$ } \end{minipage}
&  \begin{minipage}{5cm}{$(\alpha\beta-2\alpha-2 \beta+3) (\alpha-\beta) (\beta-3) (\beta+1)
(\alpha\beta^3-\alpha\beta^2-2\alpha\beta+4\alpha-\beta^3-2\beta^2-2\beta+3) \sqrt{\Delta}+
(\alpha\beta-2\alpha-2 \beta+3) (\alpha-\beta) (\beta+1) 
(\alpha^2\beta^5- \alpha \beta^5-4\alpha^2\beta^4-2\alpha\beta^4+  \beta^4
+\alpha^2\beta^3+14\alpha\beta^3+13\beta^3+10\alpha^2\beta^2-10\alpha\beta^2-10\beta^2+
4\alpha^2\beta-17\alpha\beta-15\beta-16\alpha^2-24\alpha-9)$ } \end{minipage}&  
   \begin{minipage}{5cm}{$
 -(\alpha \beta-2\alpha-2 \beta+3) (\alpha-\beta) (\beta-3) (\beta+1) (\beta^4-\alpha\beta^3-2\beta^3
   +\alpha\beta^2-\beta^2+2 \alpha \beta+\beta+4 \alpha+3)\sqrt{\Delta}  
+   (\alpha \beta-2\alpha-2 \beta+3) (\alpha-\beta) (\beta+1)
     (-\alpha \beta^6+\alpha^2\beta^5 +5\alpha\beta^5+3 \beta^5
   -4 \alpha^2\beta^4-8 \alpha\beta^4-3 \beta^4+\alpha^2\beta^3-18 \alpha\beta^3-17 \beta^3
   +18 \alpha^2\beta^2+53\alpha\beta^2+26 \beta^2-20 \alpha^2\beta+\alpha\beta+12 \beta
   -16\alpha^2-24\alpha-9)
    $ } \vspace{0.5mm}\end{minipage} \\[7pt]
 \begin{minipage}{5cm}{$ (\beta-3)^2 (\beta+1)^2
 (8 \alpha+ \beta^4-3\beta^3-3\beta^2-\beta+6) \sqrt{\Delta}  
 +(\beta-3)^2 (\beta+1)^2 \beta (-8 \alpha^2+
 \alpha\beta^4-3\alpha \beta^3-3\alpha\beta^2+23\alpha\beta+6\alpha
 -3\beta^3 -5\beta^2+15 \beta+9)  $ } \end{minipage} &  
 \begin{minipage}{5cm}{$ (\beta-3) (\beta+1) 
 (4 \beta^5-16 \beta^4+5 \beta^3+8 \beta^2-15 \beta+
  \alpha^2\beta^5-8 \alpha^2\beta^4+21 \alpha^2\beta^3-14 \alpha^2\beta^2-28\alpha^2 \beta+32\alpha^2 
  -\alpha \beta^6+7 \alpha \beta^5-20\alpha  \beta^4+30 \alpha \beta^3+\alpha \beta^2
  -41\alpha  \beta+48\alpha+18)\sqrt{\Delta} 
    +(\beta-3)(\beta+1) \beta  (2 \beta^5-22 \beta^4+28 \beta^3+9 \beta^2
   +\alpha^3\beta^5-8\alpha^3 \beta^4+21 \alpha^3\beta^3-14\alpha^3 \beta^2-12\alpha^3\beta+16)
   -\alpha^2\beta^6+\alpha^27 \beta^5-23\alpha^2 \beta^4+52\alpha^2 \beta^3
   -70\alpha^2 \beta^2-33 \alpha^2\beta+72\alpha^2
   +7 \alpha \beta^5-37\alpha  \beta^4+73 \alpha \beta^3-30 \alpha \beta^2-18 \alpha \beta+81\alpha 
   +27) $ } \end{minipage}& 
   \begin{minipage}{5cm}{$(\beta-3) (\beta+1) (\beta^6-8 \beta^5+15 \beta^4+3 \beta^3-25 \beta^2+3 \beta
   - \alpha^2\beta^5+8\alpha^2 \beta^4-21 \alpha^2\beta^3+6\alpha^2 \beta^2+36\alpha^2 \beta-16\alpha^2
   +\alpha \beta^6-9\alpha  \beta^5+28\alpha  \beta^4-18\alpha  \beta^3-33\alpha  \beta^2
   +31\alpha  \beta-24\alpha -9)  \sqrt{\Delta}     
   +(\beta-3) (\beta+1) (-7 \beta^6+26 \beta^5-23 \beta^4-25 \beta^3-3 \beta^2
   -\alpha^3 \beta^6+8\alpha^3 \beta^5-21 \alpha^3\beta^4+22 \alpha^3\beta^3+4\alpha^3
   \beta^2-32\alpha^3\beta+9 \beta
   +\alpha \beta^7-11 \alpha \beta^6+54\alpha  \beta^5-73\alpha  \beta^4-15\alpha  \beta^3
   +22\alpha  \beta^2+18 \alpha \beta+72\alpha +\alpha^2 \beta^7-9 \alpha^2 \beta^6
   +31\alpha^2  \beta^5-64 \alpha^2 \beta^4+46 \alpha^2 \beta^3+39 \alpha^2 \beta^2-16 \alpha^2 \beta
   +48\alpha^2 +27)$ \vspace{1mm}} \end{minipage} \\[7pt]
 \begin{minipage}{5cm}{$   (\beta-3)^2  (\beta+1)^2
 (2 \beta^3-\beta^2-8 \beta- \alpha\beta^4+4\alpha \beta^3-\alpha\beta^2-10\alpha \beta+4\alpha+3)
  \sqrt{\Delta}
 +(\beta-3)^2(\beta+1)^2 (2 \beta^4-8 \beta^3-5 \beta^2+6 \beta+5\alpha  \beta^4-11\alpha  \beta^3
 -13 \alpha \beta^2+11\alpha  \beta+24\alpha - \alpha^2\beta^5+4\alpha^2 \beta^4-\alpha^2\beta^3
 -10\alpha^2 \beta^2-4\alpha^2 \beta-16\alpha^2+9)  $ } \end{minipage}& 
   \begin{minipage}{5cm}{$(\beta-3) (\beta+1) 
    (3 \beta^5-14 \beta^4+22 \beta^3+7 \beta^2-39 \beta+ 2 \alpha^2\beta^5
   -13\alpha^2 \beta^4+26 \alpha^2\beta^3-5 \alpha^2\beta^2-34 \alpha^2\beta+28\alpha^2
   -\alpha \beta^6+7 \alpha \beta^5-25\alpha  \beta^4+44 \alpha \beta^3-3 \alpha \beta^2
   -79\alpha \beta+33\alpha +9) \sqrt{\Delta}  
   +(\beta-3) (\beta+1) (2 \beta^6-17 \beta^5+40 \beta^4-20 \beta^3-63 \beta^2+27 \beta
   + 2\alpha^3 \beta^6-13\alpha^3 \beta^5+26 \alpha^3\beta^4-5 \alpha^3\beta^3-34\alpha^3 \beta^2
   +12 \alpha^3\beta+16\alpha^3
   -\alpha^2 \beta^7+7\alpha^2 \beta^6-31\alpha^2 \beta^5+79\alpha^2 \beta^4
   -55\alpha^2 \beta^3-92 \alpha^2\beta^2+105 \alpha^2\beta+72\alpha^2   
   +6\alpha  \beta^6-37\alpha  \beta^5+97\alpha  \beta^4-73\alpha  \beta^3-138\alpha  \beta^2+108
   \alpha \beta+81\alpha +27) $ } \end{minipage}&  
    \begin{minipage}{5cm}{$(\beta-3) (\beta+1) 
       (\alpha^2\beta^6-7\alpha^2\beta^5+16\alpha^2\beta^4-9\alpha^2\beta^3-11\alpha^2\beta^2
   +18\alpha^2\beta-20 \alpha^2
        - \alpha\beta^7+7\alpha\beta^6-18\alpha\beta^5+17\alpha\beta^4
       +3\alpha\beta^3-13\alpha\beta^2+20\alpha\beta-39\alpha
       +3 \beta^6-13\beta^5+8\beta^4+17\beta^3+\beta^2+6\beta-18)\sqrt{\Delta}  
      +(\beta-3) (\beta+1) 
      (\alpha^3\beta^7-7\alpha^3\beta^6+16\alpha^3\beta^5-9\alpha^3\beta^4
      -11\alpha^3\beta^3+2\alpha^3\beta^2+28\alpha^3\beta-48\alpha^3
      -\alpha^2\beta^8+7\alpha^2\beta^7-21\alpha^2\beta^6+36\alpha^2\beta^5
      -31\alpha^2\beta^4+13\alpha^2\beta^3-14\alpha^2\beta^2-23\alpha^2\beta+24\alpha^2
   +6 \alpha\beta^7-34\alpha\beta^6+56\alpha\beta^5-6\alpha\beta^4-16\alpha\beta^3-86\alpha\beta^2
   -72\alpha\beta+9\alpha    
   +2 \beta^7-17\beta^6+29\beta^5+16\beta^4-13\beta^3-33\beta^2-36\beta) 
      $ } \end{minipage} 
\end{array}
\right)
$$
}

The choices $T\in \SL(3,\CC)$ are of the form $\varpi d_0^{-1/3} T_0$, $\varpi\in \mu_3$,
$d_0=\det (T_0)$. Therefore, the coordinates $y= \tr (T)$, $\bar{y}=\tr (T^{-1})$ ,
$z=\tr(TA^{-1} TA)$ and $\bar z=\tr(A^{-1} T^{-1} AT^{-1})$
can be computed using this explicit parametrization.
Noting that only quantities that are $\mu_3$-invariant can be computed, we 
get by explicit calculation \cite{fig8html}:
%are given as 
%$y=y_0d_0^{-1/3}$, $\bar{y}= \bar{y}_0 d_0^{-2/3}$.
%A computer calculation gives
% \begin{align*}
% y  \bar{y} &= (\alpha+1)(\beta+1), \\
% y^3 + \bar{y}^3 &= \alpha^2\beta+\alpha\beta^2+ 6 \alpha\beta+3 \alpha+ 3 \beta+2 .
% \end{align*}
%This gives a six-to-one covering of the plane $(\alpha,\beta)$, 
%which is a triple covering over the double covering defined by $(\alpha,\beta,\eta)$.
%
%There is an action of $\mu_3$ given by $(y,\bar{y})\mapsto (\varpi y,\varpi^2 \bar{y})$, and an action
%of $\ZZ_2$ given by $(y,\bar{y})\mapsto (\bar{y},y)$. The ramification locus of the former action is $y=\bar{y}=0$, i.e. $y^3+\bar{y}^3=y  \bar{y}=0$. This
%corresponds to $\alpha=\beta=-1$. The ramification locus of the latter is given
%by $y=\bar{y}$, that is $(y^3+\bar{y}^3)^2-4 (y  \bar{y})^3=0$. This is rewritten as
%  $$
%   (\alpha-\beta)^2\Delta=0.
%  $$ 
%Therefore for $(\alpha,\beta)$ with $\alpha\neq \beta$, there are six values
%for $(y,\bar{y})$ unless $\Delta=0$, in which case there are three values of $(y,\bar{y})$. This means
%that $(\alpha,\beta, y, \bar{y})$ parametrizes such set. For $\alpha=\beta$ however, there
%are three values of $(y,\bar{y})$ (unless $\alpha=\beta=-1$ in which case we only have $(y,\bar{y})=(0,0)$).
%So $(\alpha,\beta,y,\bar{y})$ does not parametrize the character variety on this subset, and we
%have to consider $(\alpha,\beta,\eta, y,\bar{y})$. 
%But then there must be an extra relation between $\eta,y,\bar{y}$. This is found by an
%explicit computation:
 \begin{align}\label{eq:distinguishedComp}
  y  \bar{y} &= (\alpha+1)(\beta+1), \notag\\
  2y^3 &=(\alpha^2\beta+\alpha\beta^2+ 6 \alpha\beta+3 \alpha+ 3 \beta+2)- 
  (\alpha-\beta)  \sqrt\Delta ,\notag\\
 2 \bar{y}^3 &=(\alpha^2\beta+\alpha\beta^2+ 6 \alpha\beta+3 \alpha+ 3 \beta+2) + 
  (\alpha-\beta)  \sqrt\Delta , \notag\\
  2\eta&= \tr([A,B])= (\alpha^2\beta^2-2\alpha^2\beta-2\beta^2\alpha+2\alpha^2+2\beta^2-3)
  -  (\alpha\beta-2\alpha-2\beta+3 )\sqrt{\Delta}\, , \notag\\
 2z^3 &=\alpha^4 \beta^2  + 10 \alpha^2\beta+ 9\alpha^2- 2 \alpha^3 -2
  -(-4 \alpha+3 \alpha^2+\alpha^3\beta) \sqrt{\Delta} ,\\
 2\bar z^3 &=\alpha^4 \beta^2  + 10 \alpha^2\beta+ 9\alpha^2- 2 \alpha^3 -2
  +(-4 \alpha+3 \alpha^2+\alpha^3\beta) \sqrt{\Delta} ,\notag\\
 z\bar z &= 1+\alpha^2+2\alpha^2\beta, \notag\\
 2yz &= \alpha^2\beta+3\alpha\beta+ 3\alpha+1 + (1-\alpha)\sqrt{\Delta}, \notag\\
 2\bar{y}\bar z &= \alpha^2\beta+3\alpha\beta+ 3\alpha+1- (1-\alpha)\sqrt{\Delta}.\notag
   \end{align}
This produces the equations in the statement. 
%It is easy to see that these 
%equations define a smooth variety, so the ideal they define is radical, hence
%it is the full ideal.   
%Also
% \begin{align*}
% 2(\alpha+1)(\beta+1) z &= (  \alpha^2\beta+3\alpha\beta+ 3\alpha+1 + (1-\alpha)\sqrt{\Delta}) \bar{y}  \\
% 2(\alpha+1)(\beta+1) \bar{z} &= (  \alpha^2\beta+3\alpha\beta+ 3\alpha+1 - (1-\alpha)\sqrt{\Delta})  y,
% \end{align*}
%from where $yz=\bar{y}\bar{z}$.

Now notice that the covering $V_0\to\CC^2$ given by 
$\chi  \mapsto  (\chi(a),\chi(b))$ is regular. More precisely, the group of 
deck transformations $\mathcal{D}$ is generated by $\mu_3$ and $\iota$, where
$$\iota\co X(\Gamma,\SL(3,\CC))\to X(\Gamma,\SL(3,\CC))$$ is  given by
$\iota(\chi)(\gamma) = \chi(\gamma^{-1})$. If $\chi=\chi_\rho$ is the character of the 
representation $\rho\co\Gamma\to\SL(3,\CC)$ then $\iota(\chi)=\chi_{\rho^*}\co\Gamma\to\CC$ is the
character of the dual representation $\rho^*\co\Gamma\to\SL(3,\CC)$ given by
\[ \forall \gamma\in\Gamma,\quad
\rho^* (\gamma) = (\rho(\gamma)^{-1})^\mathrm{t},
\]
where $A^\mathrm{t}$ denotes the transpose matrix for $A\in\SL(3,\CC)$.
In the coordinates $(y,\bar{y},z,\bar{z},\alpha,\beta)$ the action of $\varpi\in\mu_3$ and $\iota$ is given by
\[ \varpi (y,\bar{y},z,\bar{z},\alpha,\beta) = (\varpi y, \varpi^2\bar{y},\varpi^2z,\varpi\bar{z},\alpha,\beta) \quad
\text{ and } \quad
\iota(y,\bar{y},z,\bar{z},\alpha,\beta)= (\bar{y},y,\bar{z},z,\alpha,\beta)\,.\]
The ring of invariant functions of this action is generated by $\alpha$, $\beta$ and the following functions:
% \comment{M: I calculated this with \emph{Magma}. We should cite it. What do you think? -- It is not  obvious.}
 \[ 
 f_1 = y\bar{y} , \  f_2 = z\bar{z} , \  f_3= y^3+\bar{y}^3, \  f_4 = z^3 +\bar{z}^3, \  
h_1 = yz+\bar{y}\bar{z} , \  h_2 = \bar{y}^2z + y^2\bar{z}, \  h_3 = \bar{y}z^2 +  y\bar{z}^2\,.
\]
Therefore the quotient $V_0/\mathcal{D}$ embeds into $\CC^7\times\CC^2$ and it follows from equations \eqref{eq:distinguishedComp} that the image $q\co V_0\to\CC^7\times\CC^2$ is isomorphic  to $\CC^2$ and given by
\begin{multline*}
  f_1= (\alpha+1)(\beta+1), \
       f_2= 2\alpha^2\beta+\alpha^2+1, \
    f_3 = \alpha^2\beta+\alpha\beta^2+ 6 \alpha\beta+3 \alpha+ 3 \beta+2, \\
 f_4 = \alpha^4 \beta^2  + 10 \alpha^2\beta+ 9\alpha^2 - 2 \alpha^3 -2  ,\\
 h_1 = \alpha^2\beta+3\alpha\beta+ 3\alpha+1, 
  h_2 =  \alpha^{2} \beta^{2}  + 4
 \alpha^{2} \beta + 2 \alpha^{2} + 4 \alpha \beta + 2 \alpha + 2 \beta + 1,\\
 h_3= \alpha^{3} \beta^{2} +  \alpha^{3} \beta +  4 \alpha^{2} \beta 
 + 3 \alpha^{2} + 5 \alpha \beta + 3\alpha + 1\,.
\end{multline*}
Here $(f_1,f_2,f_3,f_4,h_1,h_2,h_3,\alpha,\beta)$ are the coordinates of $\CC^7\times\CC^2$.
Finally, $V_0=q^{-1}(V_0/\mathcal{D})$ is the zero-locus of the equations stated in the proposition.

The intersection of $V_0$ with the reducible locus is as follows:
\begin{itemize}
\item With the totally reducible representations (i.e.,\ $\alpha=\beta=3$), it is given by
$(\eta,y,\bar{y})=(3,4,4)$, $(3,4\varpi,4\varpi^2)$, $(3,4\varpi^2,4\varpi)$, $\varpi=\vvarpi$.\\
\item With the partially reducible representations (i.e.,\ $\alpha\beta-2 \alpha-2 \beta+3=0$, with $(\alpha,\beta)\neq (3,3)$), 
it is given by six points over each $(\alpha,\beta)$, if $\alpha\neq 1,3,\frac{1\pm\sqrt 5}{2}$, 
defined by the six solutions for $(y,\bar{y})$,
cf.\ Corollary~\ref{cor:fibre}.
For $(\alpha,\beta)=\left( \frac{1\pm\sqrt 5}{2},  \frac{1\mp\sqrt 5}{2} \right)$,
there are only three values $(\eta,y,\bar{y})$, where $y^3=\bar{y}^3=-1$ 
and $y \bar{y}=1$, $\eta=3$.
For $\alpha=\beta=1$, we have $\eta=-1$, 
$(y,\bar{y})=(2,2),(2\varpi,2\varpi^2),(2\varpi^2,2\varpi)$. These are
the same three points as in $V_j \cap X_\mathrm{PR}$, $j=1,2$.
\end{itemize}
\end{proof}

\begin{remark}\label{rem:radical}
We know by  Proposition~\ref{prop:smooth} that the component $V_0$ is scheme reduced.
 On the other hand, it is easy to see that the 
ideal generated by the equations in the statement of Proposition~\ref{prop:distinguishedComp} is not radical reduced i.e.\ 
the scheme is not reduced. 
Now computer supported calculations \cite{fig8html} produce %in less than a minute 
generators of the corresponding radical, i.e.\ generators of the vanishing ideal $I(V_0)$.
More precisely, $I(V_0)\subset\CC[y,\bar y,z,\bar z,\alpha,\beta]$ 
is generated by the following 18 polynomials:
\begin{gather*}
 2 y z + 2 \bar{y} \bar{z} -  z \bar{z} + \alpha^{2} - 6\alpha \beta - 6 \alpha - 1 ,\quad
  y \bar{y} -  \alpha \beta -  \alpha -  \beta - 1 ,\quad
  2 \alpha^{2} \beta -  z \bar{z} + \alpha^{2} + 1 ,\\
  \bar{y} \alpha \beta + y^{2} -  y \bar{z} + \bar{y} \alpha-  z \beta -  z ,\quad
  y \alpha \beta + \bar{y}^{2} -  \bar{y} z + y \alpha - \bar{z} \beta -  \bar{z} ,\\
  y \bar{z} \alpha -  \bar{y} \alpha^{2} -  \bar{z}^{2} + 3\bar{y} \alpha -  z \alpha -  z ,\quad
  \bar{y} z \alpha -  y \alpha^{2} -  z^{2} + 3 y \alpha - \bar{z} \alpha -  \bar{z} ,\\
  \bar{y}^{2} \alpha -  \bar{y} z + y \alpha -  \bar{z} \alpha + y -  \bar{z} ,\quad
  y^{2} \alpha -  y \bar{z} + \bar{y} \alpha -  z \alpha +
\bar{y} -  z ,\\
  2 \bar{y} \bar{z}^{2} -  z \bar{z}^{2} + 2 y \alpha^{2} +
\bar{z} \alpha^{2} - 2 \bar{z} \alpha \beta - 4 \bar{y} z +
4 z^{2} - 8 y \alpha - 2 \bar{z} \alpha + 6 y -  \bar{z} ,\\
  \bar{y} z \bar{z} -  \bar{y} \alpha^{2} - 2 z \alpha \beta
+ 2 y \bar{z} - 2 \bar{z}^{2} + 4 \bar{y} \alpha - 2 z
\alpha - 3 \bar{y} ,\\
  \bar{y}^{2} \bar{z} + 3 y^{2} - 2 y \bar{z} - 
\bar{z}^{2} + 2 \bar{y} \alpha - 2 z \beta - 2 \bar{y} - 2 z
,\\
  4 \bar{y}^{2} z + 4 y^{2} \bar{z} - 2 z \bar{z} \beta - 7
z \bar{z} -  \alpha^{2} - 16 \alpha \beta - 8 \alpha - 6 \beta + 3
,\\
  \bar{y}^{3} + y^{2} \bar{z} -  \bar{y} \bar{z} \beta +
\alpha \beta^{2} -  \bar{y} \bar{z} -  z \bar{z} - 2 \alpha
\beta -  \alpha - 2 \beta ,\\
  2 y^{3} - 2 y^{2} \bar{z} + 2 \bar{y} \bar{z} \beta - 4
\alpha \beta^{2} + 2 \bar{y} \bar{z} + z \bar{z} +
\alpha^{2} - 8 \alpha \beta - 4 \alpha - 2 \beta - 3 ,\\
  2 z \bar{z} \alpha \beta - 4 \bar{y} z^{2} - 4 y
\bar{z}^{2} + z \bar{z} \alpha -  \alpha^{3} + 8 z \bar{z} +
4 \alpha^{2} + 18 \alpha \beta + 11 \alpha - 12 ,\\
  z^{2} \bar{z}^{2} - 2 z \bar{z} \alpha^{2} + \alpha^{4} - 4
z^{3} - 4 \bar{z}^{3} - 8 \alpha^{3} + 18 z \bar{z} + 18
\alpha^{2} - 27 ,\\
  4 y^{2} \bar{z}^{2} - 2 z \bar{z}^{2} \beta - 3 z
\bar{z}^{2} -  \bar{z} \alpha^{2} + 8 z^{2} \beta + 36
\bar{y}^{2} - 32 \bar{y} z + 4 z^{2} + 16 y \alpha - 30
\bar{z} \beta - 29 \bar{z}\,.
\end{gather*}
\end{remark}

\begin{remark}
By Proposition~\ref{prop:smooth}, the only singular points of $V_0$ are
$\mu_3 \cdot (2,2,2,2,1,1,1,1)$. Using equations (\ref{eq:distinguishedComp}),
we can parametrize $(y,\bar{y},z,\bar{z})$ locally around $\alpha=2,\beta=2$.
There are two branches, depending on the choice of sign for $\sqrt{\Delta}$. 
Both branches are smooth and intersect transversely at the point. So the
singularities are simple nodes.
\end{remark}

We have also characterized the partially reducible  representations 
and the totally reducible representations that can deform into
irreducible ones (see also the description in Lemma~\ref{lemm:redtoirreduc}). 

\begin{remark}
In Proposition \ref{prop:parameters}, we have said that $V_0$ can be described
pointwise by the parameters $(\alpha,\beta,\eta, y,\bar{y})$. The equations are
\cite{fig8html}:
  \begin{align*}
  &  \eta^2 - P \eta + Q =0, \ \text{where $P,Q$ are given in Lemma \ref{lem:Wj},}\\
 &  y  \bar{y} = (\alpha+1)(\beta+1), \\
   & y^3 + \bar{y}^3 = \alpha^2\beta+\alpha\beta^2+ 6 \alpha\beta+3 \alpha+ 3 \beta+2, \\
   & (\alpha-\beta) (2\eta-(\alpha^2\beta^2-2\alpha^2\beta-2\beta^2\alpha+2\alpha^2+2\beta^2-3))= \\
    & \qquad \qquad \qquad \qquad =(\alpha\beta-2\alpha-2\beta+3 )
(2y^3-(\alpha^2\beta+\alpha\beta^2+ 6 \alpha\beta+3 \alpha+ 3 \beta+2)).
  \end{align*}
\end{remark}

%%%%%%%%%%%%%%%%%%%%%%%%%%%%%%%%%%%%%%%%%%%%%
\section{Character varieties for $\PGL(3,\CC)$ and $\GL(3,\CC)$} \label{sec:irredPGL3GL3}
%%%%%%%%%%%%%%%%%%%%%%%%%%%%%%%%%%%%%%%%%%%%%

We use the descriptions of the various components of the $\SL(3,\CC)$-character
variety given in Sections \ref{sec:reducible}, \ref{sec:nondist} and \ref{sec:distinguished}
to give a similar description for the character varieties for $\PGL(3,\CC)$ and
$\GL(3,\CC)$.

\begin{proposition}
 The character variety $X(\Gamma,\PGL(3,\CC))$ has five components.
  \begin{itemize}
  \item The component corresponding to totally reducible representations, isomorphic
  to $\CC^2/\mu_3$, $\mu_3$ acting by $(y,\bar{y})\mapsto (\varpi y ,\varpi^2 \bar{y})$.
  \item The component corresponding to partially reducible representations, isomorphic
  to $$\{(x_1,x_2,v,w) \mid x_1x_2=x_1+x_2, (x_1+x_2+1)w= v^2, w\neq 0\}/\mu_3,$$
  $\mu_3$ acting by $(v,w)\mapsto (\varpi v, \varpi^2 w)$.
  \item Three components consisting of irreducible representations, which
  are $W_0- \{\alpha\beta-2\alpha-2\beta+3=0 \} , W_1-\{\beta=\bar\beta=1\}, W_2-\{\alpha=\bar\alpha=1\}$ defined in Lemmas \ref{lem:Wj} and \ref{lemma:intersect}. 
 \end{itemize}
 \end{proposition}  

\begin{proof}
The $\PGL(3,\CC)$-character variety is obtained from the $\SL(3,\CC)$ by taking
the quotient by $\mu_3$ acting on the coordinates $y$, $\bar y$, $z$ and $\bar z$.
The locus of reducible representations is determined in Proposition \ref{prop:SL3-reducibles}.
For the partially reducible representations, we use the description in (\ref{eqn:PR}).

Finally, from the discussion at the end of Section \ref{sec:Lawton}, the
irreducible locus
$X_{irr}(\Gamma,\PGL(3,\CC))$ is isomorphic to the image
under the restriction map 
$\res\co X_{irr}(\Gamma,\GL(3,\CC))\to X(F_ 2,\SL(3,\CC))$. 
This is the set described in Lemmas \ref{lem:Wj} and \ref{lemma:intersect}.
\end{proof}

To have a coordinate description that allows to understand the intersection of
the closure of  $X_{irr}(\Gamma,\PGL(3,\CC))$ with the other components, 
one can get a coordinate description introducing new variables
which are the generators of the ring $\CC[y,\bar{y},z,\bar{z}]^{\mu_3}$,
namely
$u_1=y^3, u_2=\bar{y}^3, u_3=y\bar{y}, u_4=z^3, u_5=\bar{z}^3, 
u_6=z\bar{z},u_7=y\bar{z}, u_8=\bar{y}z,
u_9=y\bar{z}^2 , u_{10}=y^2 \bar{z}, u_{11}=\bar{y} z^2, u_{12}=\bar{y}^2z$. 
They satisfy certain relations which are easy to write down. A
substitution of these variables into the statement of Theorem \ref{thm:main} 
gives the equations that they satisfy.

To get the $\GL(3,\CC)$-representations,  recall that
 $$
 X(\Gamma,\GL(3,\CC)) = (X(\Gamma,\SL(3,\CC)) \x \CC^*)/\mu_3 \, .
  $$
So the variables are $(y,\bar{y},z,\bar{z},\alpha,\bar{\alpha},\beta,\bar{\beta},\lambda)$
with $\mu_3$ acting as 
 $$
 \varpi\cdot(y,\bar{y},z,\bar{z},\alpha,\bar{\alpha},\beta,\bar{\beta},\lambda)=
(\varpi y,\varpi^2\bar{y},\varpi^2z,\varpi\bar{z},\alpha,\bar{\alpha},\beta,\bar{\beta},\varpi\lambda).
$$
Thus one can get a coordinate description introducing the variables
$u_1=y^3, u_2=\bar{y}^3, u_3=y\bar{y}, u_4=z^3, u_5=\bar{z}^3, 
u_6=z\bar{z},u_7=y\bar{z}, u_8=\bar{y}z,
u_9=y\bar{z}^2 , u_{10}=y^2 \bar{z}, u_{11}=\bar{y} z^2, u_{12}=\bar{y}^2z$,
and $v_1=y\lambda^2,v_2=y^2\lambda, v_3=\bar{y}\lambda,
v_4=z\lambda,v_2=\bar{z}\lambda^2, v_3=\bar{z}^2\lambda$.
They satisfy certain relations which are easy to write down. Substituting 
these variables into the statement of Theorem \ref{thm:main}
gives the equations that they satisfy.

For the totally reducible representations, we have a simpler description
as $X_\mathrm{TR}(\Gamma,\GL(3,\CC)) = \CC^2\x \CC^*$. Also, for the 
partially reducible $\GL(3,\CC)$-representations, they 
split as an irreducible $\GL(2,\CC)$-representation and a one-dimensional
representation. So
  \begin{align*}
 X_\mathrm{PR}(\Gamma,\GL(3,\CC)) &= X_{irr}(\Gamma,\GL(2,\CC)) \x  \CC^*\, .
 \end{align*}

%%%%%%%%%%%%%%%%%%%%%%%%%%%%%%%%%%%%%%%%%%%%%
\section{The symmetry group of the figure eight knot}\label{sec:symmetries}
%%%%%%%%%%%%%%%%%%%%%%%%%%%%%%%%%%%%%%%%%%%%%

% % We describe the symmetry group using the Wirtinger presentation 
% %  \begin{equation}\label{eq:Wirtinger}
% %  \Gamma=\langle S,T\mid S T^{-1} S^{-1} T  S = T  S T^{-1} S^{-1} T\rangle\,.
% %  \end{equation}
% %  It is related to 
% % the presentation
% % $
% %  \Gamma =\la a,b,t \mid tat^{-1} =ab , t b t^{-1}=bab \ra
% % $ 
% %  by putting 
% % \[
% % \begin{cases}
% % S&= t \\
% % T&= a^{-1} t  a,
% % \end{cases} 
% % \quad\text{ and }\quad
% % \begin{cases}
% % t&= S \\
% % a&= T^{-1}STS^{-1}\\
% % b&= T S^{-1}.
% % \end{cases} 
% %  \]

The symmetry group $\mathrm{Sym}(S^3,K_8)$ of the figure eight knot $K_8$ is isomorphic to the outer automorphism group 
\[
\mathrm{Out}(\Gamma) = \mathrm{Aut}(\Gamma)/\mathrm{Inn}(\Gamma)
\]
(see \cite[10.6]{Kawauchi1996}). The group $\mathrm{Out}(\Gamma)$ was calculated by Magnus \cite{Magnus1931} %in 1931 
(see also \cite{Tsau1986}). It
 is isomorphic to the dihedral group $D_4$ of order eight
\begin{equation}\label{eq:OutGamma}
 \mathrm{Out}(\Gamma) = \la f,h \mid f^2 = h^4 = (fh)^2=1\ra\cong D_4\, ,
\end{equation}
where the elements $f$ and $h$ are represented by the following automorphisms 
(also denoted by $f$ and $h$):
$$ 
 \begin{array}{l}f(S) = T^{-1} \\[3pt]  f(T) = S^{-1} \end{array}
 \qquad \text{and}\qquad 
 \begin{array}{l} h(S) = S  T^{-1} S^{-1} \\[3pt]  h(T) = T  S^{-1} T^{-1}\,. \end{array}
$$
They are also described by the action on  $t,a,b\in \Gamma$ as follows:
\begin{align}
f(t) &= T^{-1} =  a^{-1} t^{-1} a  \sim t^{-1}\notag\\ 
f(a) & = ST^{-1}S^{-1}T   =   a^{-1} \label{eq:f} \\
f(b) &= S^{-1}T= b a^{-1}  \sim b \notag
\end{align}
and
\begin{align}
h(t) &= ST^{-1}S = t a^{-1} t^{-1} a t^{-1} \sim t^{-1}\notag\\ 
h(a) & = TST^{-2} = T b^{-1} T^{-1} = a^{-1} t a  b^{-1}  a^{-1} t^{-1} a \sim b^{-1} \label{eq:h} \\
h(b) &= TS^{-1}T^{-1} STS^{-1}= STS^{-1}T^{-1}= T a T^{-1} = a^{-1} t a  t^{-1} a\sim a \notag
\end{align}

A peripheral system $(m,\ell)$ of the figure eight knot is given by
 \begin{equation}\label{eq:peripheral}
  \begin{array}{l}  
      m = S = t\, , \\[3pt] %\text{ and } 
      \ell = T^{-1} S T S^{-1}  S^{-1} T  S  T^{-1} = [a,b]\,.
  \end{array}
\end{equation}
Notice that by \eqref{eq:h} we obtain 
 \begin{align*}
 h(m) &= t a^{-1}  m^{-1}   t^{-1} a, \\
h(\ell) &= h([a,b]) = a^{-1} t a   [b^{-1},a]   a^{-1} t^{-1} a\,.
 \end{align*}
Now the relation $t^{-1} a^{-1} t = b a^{-2}$ gives that
the peripheral system $(h(m),h(\ell))$ is conjugated to $(m^{-1},\ell)$. This reflects the amphicheirality of the figure eight knot. 

The induced action on the varieties of representations are given in coordinates as follows:
$$
f^*(y)=\bar y, \ f^*(\bar y)=y, \ f^*(z)=\bar z, \ f^*(\bar z)=z,    \ f^*(\alpha)=\bar\alpha, \  f^*(\bar \alpha)=\alpha, \ f^*(\beta)=\beta, \  f^*(\bar \beta)=\bar\beta,
$$
and
$$
h^*(y)=\bar y, \ h^*(\bar y)=y, 
 h^*(\alpha)=\bar\beta, \  h^*(\bar \alpha)=\beta, \ h^*(\beta)=\alpha, \  h^*(\bar \beta)=\bar\alpha.
$$

\begin{lemma} $h^*(z)= \bar y^2-\bar z$ and $h^*(\bar z)=  y^2- z$. \end{lemma}

\begin{proof} 
We check this equality on each component. Away from $V_0$, $z$ and $\bar z$ are functions on the other variables, 
and the proof is straightforward. For $V_0$, we compute $ h^*(z)$ as follows:
we have $ta^{-1}ta = ST$ and
\[ h(ST) = S T^{-1} S^{-1} T S^{-1} T^{-1} = T^{-1} S T^{-1}S^{-1}\sim T^{-1}S^{-1} T^{-1} S\,.\]
Now, we proceed as in the proof of Lemma~\ref{lem:projection}. 
Let  $\rho\co\Gamma\to\SL(3,\CC)$ be a representation. We put 
$\rho(S)=M$ and $\rho(T)=N$. By the Cayley--Hamilton theorem  we have
\[ (N^{-1}M^{-1})^3 = \tr(N^{-1}M^{-1}) (N^{-1}M^{-1})^2 - \tr(MN) (N^{-1}M^{-1}) + \mathrm{id}\,.\]
Multiplying this identity by $MNM^2$ gives:
\[ \tr(N^{-1}M^{-1}N^{-1}M) = \tr(N^{-1}M^{-1}) \, \tr(N^{-1}M) - \tr(MN)\,\tr(M^{2}) +\tr(M^3 N)\,.\]
Applying the same procedure to $M^3 N$ and $M^2 N$, we obtain
\begin{align*}
\tr(M^3 N) &= \tr(M) \, \tr(M^{2}N) - \tr(M^{-1})\,\tr (MN) + \tr(N) 
\intertext{and}
\tr(M^2 N) &= \tr(M) \, \tr(MN) - \tr(M^{-1})\,\tr (N) + \tr(M^{-1}N)\,.
\end{align*}
Now $S=t$, $T=a^{-1}ta$, $ta^{-1}ta = ST$ and $b = TS^{-1}$ gives 
$h^*(z) = \bar z \bar \beta + \bar y z - y^2 \bar y  + y\beta +y$.
Using that on $V_0$ $\bar \alpha=\alpha$
and $\bar \beta=\beta$, 
$y\bar y= (1+\alpha)(1+\beta)$, and 
$  y \alpha \beta + \bar{y}^{2} -  \bar{y} z + y \alpha - \bar{z} \beta -  \bar{z}=0$
(see Remark~\ref{rem:radical})
the computation for $ h^*(z)$ follows.
The formula for $ h^*(\bar z) $ is proved in the same way.
\end{proof}

% For the coordinates $z$ and $\bar z$ we have:
% \begin{align*}
%  h^*(z) &= \bar z \bar \beta + \bar y z - y^2 \bar y  + y\beta +y\, ;\\
%  h^*(\bar z) &=  z  \beta +  y \bar z - y \bar y^2 + \bar y\bar \beta + \bar y \,.
%  \end{align*}
%  \end{lemma}
% \begin{proof} 
% In order to calculate $ h^*(z)$ we proceed as follows: we have $ta^{-1}ta = ST$ and
% \[ h(ST) = S T^{-1} S^{-1} T S^{-1} T^{-1} = T^{-1} S T^{-1}S^{-1}\sim T^{-1}S^{-1} T^{-1} S\,.\]
% Now, we proceed as in the proof of Lemma~\ref{lem:projection}. 
% Let  $\rho\co\Gamma\to\SL(3,\CC)$ be a representation. We put 
% $\rho(S)=M$ and $\rho(T)=N$. By the Cayley--Hamilton theorem  we have
% \[ (N^{-1}M^{-1})^3 = \tr(N^{-1}M^{-1}) (N^{-1}M^{-1})^2 - \tr(MN) (N^{-1}M^{-1}) + \mathrm{id}\,.\]
% Multiplying this identity by $MNM^2$ gives:
% \[ \tr(N^{-1}M^{-1}N^{-1}M) = \tr(N^{-1}M^{-1}) \, \tr(N^{-1}M) - \tr(MN)\,\tr(M^{2}) +\tr(M^3 N)\,.\]
% Applying the same procedure to $M^3 N$ and $M^2 N$, we obtain
% \begin{align*}
% \tr(M^3 N) &= \tr(M) \, \tr(M^{2}N) - \tr(M^{-1})\,\tr (MN) + \tr(N) 
% \intertext{and}
% \tr(M^2 N) &= \tr(M) \, \tr(MN) - \tr(M^{-1})\,\tr (N) + \tr(M^{-1}N)\,.
% \end{align*}
% Now $S=t$, $T=a^{-1}ta$, $ta^{-1}ta = ST$ and $b = TS^{-1}$ gives the formula for 
% $ h^*(z)$. The formula for $ h^*(\bar z) $ is proved in the same way.\\
% \comment{M: I hope everything is alright now\ldots}
% \end{proof}
Thus we have:

\begin{proposition}
$f^*$ preserves the components of $X(\Gamma,\SL(3,\CC))$ and $h^*$ swaps $V_1$ and $V_2$. 
\end{proposition}

\begin{remark}
If we consider also the action of $\mu_3$, the center of $\SL(3,\CC)$, we 
realize that $y \bar{y}$ and $y^3+\bar{y}^3$ are invariant by  both
the symmetry group of the knot and the action of $\mu_3$. This explains why 
$y \bar{y}$ and $y^3+\bar{y}^3$
are symmetric polynomials on $\alpha$ and $\beta$ for $V_0$,
on $\beta$ and $\bar \beta$ for $V_1$ and on $\alpha$ and $\bar \alpha$ for $V_2$, as well as
the symmetries on those variables (swapping $V_1$ and $V_2$). 
Similar considerations with the variables $z$ and $\bar z$ can be made.
\end{remark}

%%%%%%%%%%%%%%%%%%%%%%%%%%%%%%%%%%%%%%%%%%%%%
 \section{The non-distinguished components as Dehn fillings}
%%%%%%%%%%%%%%%%%%%%%%%%%%%%%%%%%%%%%%%%%%%%%
\label{section:dehn}

In this section we view the non-distinguished component
as the variety of representations induced by an exceptional Dehn filling 
on the figure eight knot. 
Those representations also appear in work of Bing and Martin 
\cite{Bing-Martin} in the proof of property P for twist knots. This is explained at the end of the section.

We recall that the figure eight knot has six slopes 
$s\in \QQ \cup\{\infty\}$ whose Dehn filling $K(s)$
is a \emph{small} Seifert fibered manifold; namely, a Seifert fibered manifold with 
basis orbifold a 2-sphere with 
three cone points of order $p$, $q$, and $r\geq 2$, $S^2(p,q, r)$. The precise 
coefficients are (cf. \cite{Gor}):
$$
\begin{array}{l}
K(\pm 1) \textrm{ fibers over }  S^2(2,3,7), \\
K(\pm 2) \textrm{ fibers over } S^2(2,4,5), \\
K(\pm 3) \textrm{ fibers over } S^2(3,3,4).
\end{array}
$$
The center of $\pi_1(K(\pm s))$, $s=1,2,3$, is generated by a regular fibre.
By Schur's lemma, any irreducible representation of $\pi_1(K(\pm s))\to\SL(3,\CC)$, $s=1,2,3$, 
maps the fibre to the center of $\SL(3,\CC)$.
This motivates the  study of representations of the orbifold fundamental groups 
$\pi_1^\mathcal{O}( S^2(p,q,r))$, that are
isomorphic to the (orientable) triangle groups
\[
\pi_1^\mathcal{O}( S^2(p,q,r))\cong D(p,q,r) = \la k,l \mid k^p,\ l^q,\ (kl)^r\ra\,.  
\]
These surjections of $\pi_1(K(\pm s))$, $s=1,2,3$, onto the corresponding triangle group are given by taking the quotient of the fundamental group of the small Seifert fibered manifold $\pi_1(K(\pm s))$ by its center.

In particular, using the  Wirtinger presentation \eqref{eq:pres-2bridge}, an epimorphism 
$\phi\co\Gamma\to D(3,3,4)=\langle k,l \mid l^3,\ k^3,\ (kl)^4 \rangle$
 is given by
\[
\phi(S) = k l k\ \text{  and }\ \phi(T)=k l k l k\,.
\]
With Presentation~\eqref{eq:parametersfib}, 
\begin{equation} \label{eq:phi_a-phi_b}
\phi(a)= k^{-1} l^{-1} k l , \quad \phi(b)=\phi(TS^{-1}) = k l \quad \text{ and }\quad  \phi(t) = k l k \, .
\end{equation}
It satisfies $\phi(b)^4=1$ and 
$\phi(m^3 \ell) =1$. Notice that the surjection $\phi$ induces an injection
 $$
 \phi^*\co X(D(3,3,4),\mathrm{SL}(3,\CC) )\hookrightarrow X(\Gamma,\mathrm{SL}(3,\CC) ).
 $$

Characters $\chi$ in $V_1$ satisfy $\chi(b^{\pm 1})=1$.
In addition, by \eqref{eq:phi_a-phi_b}, $\phi(b)=kl$ has order $4$. This motivates the following lemma.

\begin{lemma}\label{lem:D334}
The variety $\overline{X^{irr}(D(3,3,4),\mathrm{SL}(3,\CC) )}$ has a component $\mathcal{W}$ of dimension $2$ and three isolated points.
The variety $\mathcal{W}$ is isomorphic to the hypersurface in $\CC^3$ given by the equation
$$
\zeta^2-(\nu \bar{\nu}-2) \zeta+\nu^3+\bar{\nu}^3-5\nu\bar{\nu}+5=0\,.
$$
Here, the parameters are $\nu=\chi(k^{-1}l)$, $\bar{\nu}=\chi(k l^{-1})$ and $\zeta = \chi([k,l])$.
For every $\chi\in \mathcal{W}$, $\chi(k^{\pm 1})=\chi(l^{\pm 1})=0$ and $\chi( (kl)^{\pm 1})=1$.
 
Moreover, all characters in $\mathcal{W}$ are irreducible except for the three points
$(\nu,\bar{\nu},\zeta)=(2,2,1)$, $(2\varpi, 2\varpi^2, 1)$, $(2\varpi^2,2\varpi, 1)$, $\varpi=\vvarpi$.
\end{lemma}

\begin{proof}
For an irreducible representation of $D(3,3,4)$, the eigenvalues of the image of elements of order three  $k$ and $l$ are $\{1,\varpi,\varpi^2\}$ %, with  $\omega^2+\omega+1=0$
(otherwise the image of $k$ or $l$ would be central and the representation reducible). In particular    $\chi(k^{\pm 1})=\chi(l^{\pm 1})=0$.
For the image of $kl$ there are three possible set of eigenvalues: $\{1, i , -i\}$, $\{-1, i , i\}$,  
$\{-1, -i , -i\}$, and $\{1,-1,-1\}$. We shall see that for  $\{1, i , -i\}$ we get a two dimensional variety and for $\{-1, i , i\}$,  
$\{-1, -i , -i\}$, and $\{1,-1,-1\}$, isolated points.

First assume that the eigenvalues of the image of $kl$ are $\{1, i , -i\}$, namely $\chi(kl)=\chi( (kl)^{-1})=1$. Then, by applying Lawton's theorem and by taking 
coordinates
$\nu(\chi)=\chi(k^{-1}l)$, $\bar\nu(\chi)=\chi(kl^{-1})$, and $\zeta(\chi)=\chi([k,l])$, we get the hypersurface of $\CC^3$
\[\zeta^2-( \nu\bar\nu -2) t+\nu^3+\bar\nu^3-5\nu\bar\nu+5=0\, ,\]
that we denote $\mathcal{W}$.

Next we deal with the case where the eigenvalues of a representation of $kl$ are $\{-1, i, i\}$, namely $\chi(kl)=-1+2 i$ and $\chi( (kl)^{-1})=-1-2 i$. We apply Lawton's theorem
again, but this is not sufficient to determine a representation, as the image of $kl$ could not diagonalize. We need to impose further conditions that determine the value of the characters 
at $kl^{-1}$ and $k^{-1} l$, which will imply that the dimension of the component of the character variety is zero.
Namely, denote by $K$ and $L$ the respective images of $k$ and $l$ by a representation. As we require that $KL$ is diagonalizable, we have
\[
 0=(KL+\Id)(KL-i\Id)= (KL)^2+(1-i)KL - i \Id\, .
\]
Equivalently, $KL+(1-i)\Id - i (KL)^{-1}=0$. Multiplying with $K^{-1}$, $K L K^{-1}+(1-i)K^{-1} - i L^{-1}K ^{-2}=0$,
and since  $\tr L=\tr K^{-1}=0$, we get $\tr L^{-1}K ^{-2}=0$. In addition, since
$K^{-2}=K$, $\tr K L^{-1}=\tr L^{-1}K ^{-2}=0$. Similarly, $\tr K^{-1} L=0$, thus we get a zero dimensional variety.
Lawton's formulas \cite{Lawton0} yield that the trace of the commutator and its inverse are the same: $\tr [K,L]=\tr [K^{-1},L^{-1}]=1$. 
Thus, this is a single point in the character variety, by Lawton's coordinates. 
The case where the eigenvalues are $\{-1,- i, -i\}$ is precisely the same computation, by considering complex conjugation, and the case 
 $\{1,-1,-1\}$ is completely analogous.

Finally, let $\rho$ be a reducible semisimple representation with character $\chi_\rho$ in $\mathcal{W}$.
Hence, up to conjugation we assume that $\rho= \rho_1\oplus\rho_2$, 
 where $\rho_1\co D(3,3,4)\to\GL(2,\CC)$
 is irreducible, and $\rho_2\co D(3,3,4)\to \mu_3$ satisfies $\rho_2(g)\, \det(\rho_1(g)) =1$ for all $g\in D(3.3.4)$.
 First let us assume that $\det\circ\rho_1 = \rho_2$ is trivial, i.e.\ $\rho_1\co D(3,3,4)\to\SL(2,\CC)$.
 This implies that $\tr(\rho(g))=\tr(\rho_1(g)) +1$ for all $g\in D(3,3,4)$.
 Hence $K_1=\rho_1(k)$ and $L_1=\rho_1(l)$ are matrices of order three without common eigenspaces. Hence, up to conjugation, we can assume that
 \[
 K_1 = \begin{pmatrix} \omega & 0 \\ a &\omega^{-1}\end{pmatrix}
 \quad \text{ and } \quad
L_1 = \begin{pmatrix} \omega & 1 \\ 0 &\omega^{-1}\end{pmatrix}\,.
 \]
 Now, the condition $\tr (KL)=1$ implies $\tr(K_1L_1)=0$ and hence $a=1$.
 This gives
 \[
 K_1L_1^{-1} = \begin{pmatrix} 1 & -\omega \\ \omega^{-1} & 0\end{pmatrix},\quad
 K_1^{-1}L_1 = \begin{pmatrix} 1 & \omega^{-1} \\  -\omega & 0\end{pmatrix} \text{ and }
 [K_1,L_1] = \begin{pmatrix} 0 & 1 \\  -1 & 0\end{pmatrix}\,.
 \]
Hence, $\tr (\rho(kl^{-1}))=2$, $\tr (\rho(k^{-1}l)) =2$ and $\tr(\rho([k,l])) = 1$.

If $\lambda=\det\circ\rho_1$ is a non-trivial homomorphism, then 
$\lambda\cdot\rho\co D(3,3,4)\to \SL(3,\CC)$ is still a reducible representation
$\lambda\cdot\rho = (\lambda\cdot\rho_1) \oplus (\lambda\cdot\rho_2)$. Now,
$\lambda\cdot\rho_2$ is trivial and the preceding argument applies to $\lambda\cdot\rho$.
Finally, $\Hom(D(3,3,4),\mu_3) \cong \ZZ/3\ZZ$ and $\lambda(kl^{-1})= \lambda(k^{-1}l)^2$ implies the result.
\end{proof}

\begin{remark}
Further details in the proof of  Lemma~\ref{lem:D334} allow to describe those  
three  isolated points.
Composing with $\phi^*$, they correspond to the points in $X(\Gamma,\SL(3,\CC))$ 
with coordinates:
$$(\alpha,\bar\alpha,\beta,\bar\beta)=
(1,1,-1+2 i,-1-2 i),\ (1,1,-1-2 i,-1+2 i), \textrm{ and }(-1,-1,-1,-1).
$$
For those characters of $\Gamma$, $y=\bar y=z=\bar z=0$. Those  are 
precisely the three metabelian irreducible characters of $\Gamma$
that do not lie in $V_2$, see Corollary~\ref{cor:redired}.
\end{remark}

% \begin{remark}
% The zero-dimensional components of $\overline{X^{irr}(D(3,3,4),\mathrm{SL}(3,\CC) )}$ are exactly three points. This can be seen by calculation $P$ and $Q$ which gives for the parameters in question
% $P=\pm2$ and $Q=1$. Moreover this three points are exactly the characters of irreducible metabelian representations of $D(3,3,4)$.\\
% {\color{blue} The explicit computations are actually quite long. One needs that the commutator subgroup of $D(3,3,4)$ is generated by $kl$, $lk$ and $l^{-1}kl^{-1}$ (Reidemeister--Schreier).
% Then one has to figure out the matrices (Section~\ref{sec:explicite}) and to verify that the matrices commute. For example for the point $\{1,-1,-1\}$ we obtain
% \[ K =\left(\begin{array}{rrr}
%  0 & 0 & 1 \\
%  1 & 1 & 1 \\
%  -1 & 0 & -1
%  \end{array}\right) \text{ and } L= \left(\begin{array}{rrr}
%  1 & 0 & 1 \\
%  0 & -1 & -1 \\
%  0 & 1 & 0
%  \end{array}\right)
%  \]
%  and
%  \[
%  LK = \left(\begin{array}{rrr}
%  -1 & 0 & 0 \\
%  0 & -1 & 0 \\
%  1 & 1 & 1
%  \end{array}\right),\quad KL = \left(\begin{array}{rrr}
%  0 & 1 & 0 \\
%  1 & 0 & 0 \\
%  -1 & -1 & -1
%  \end{array}\right) \text{ and } L^{-1}KL^{-1} =
%  \left(\begin{array}{rrr}
%  0 & -1 & 0 \\
%  -1 & 0 & 0 \\
%  0 & 0 & -1
%  \end{array}\right)\,.
%  \]
%  These guys form a Klein four group. I do not think that we will enter into this\ldots
% }
% \end{remark}

\begin{proposition}\label{prop:non-dist-components}
 The components $V_1$ and $V_2$ are characters of representations which factor
 through the surjections $\Gamma\twoheadrightarrow \pi_1(K(\pm3))$ respectively.
These components are isomorphic to the hypersurface
 $$
 \zeta^2-(\nu \bar{\nu}-2) t + \nu^3+\bar{\nu}^3-5\nu\bar{\nu}+5=0.
 $$
 Here, the parameters are %$y=\chi(m)$, $\bar{y}=\chi(m^{-1})$ and 
 \[ 
 \nu = 
 \begin{cases}
\chi(t) &\text {for $V_2$,}\\
\chi(t^{-1}) &\text {for $V_1$;}
 \end{cases}
\quad
 \bar{\nu} = 
 \begin{cases}
\chi(t^{-1}) &\text {for $V_2$,}\\
\chi(t) &\text {for $V_1$;}
 \end{cases}
 \quad
  \zeta = 
 \begin{cases}
\chi(a) &\text {for $V_2$,}\\
\chi(b^{-1}) &\text {for $V_1$.}
 \end{cases}
 \]
All characters are irreducible except for the three points
 $(\nu,\bar{\nu},\zeta)=(2,2,1),\ (2\varpi, 2\varpi^2, 1), (2\varpi^2,2\varpi, 1)$, with $\varpi=\vvarpi$,
 that correspond to the intersection $V_1 \cap V_2=V_0\cap V_1 \cap V_2$.
 The intersection of $V_i\cap V_0$ is the zero locus of the discriminant on $\zeta$:
 $$
 \nu^2\bar{\nu}^2-4 \nu^3-4 \bar{\nu}^3+16 \nu \bar{\nu} -16=0.
 $$
 The restriction map $X(\Gamma,\SL(3,\CC))\to X(F_2,\SL(3,\CC))$ maps the intersection $V_1 \cap V_2$
onto a single point $\alpha=\bar\alpha=\beta=\bar\beta=1$.
\end{proposition}

% Of course this is compatible with the description in Section~\ref{sec:nondist}.

\begin{proof}
By \eqref{eq:phi_a-phi_b}, the surjection 
$\phi\co \Gamma\to D(3,3,4)$ maps $b$ to $kl$, and $a$ to a 
conjugate to $[k,l]$. Hence for every character $\chi\in \mathcal{W}$ we obtain
that 
$\res \circ\phi^*(\chi)=
\res  (\chi \circ \phi)\co F_2\to\CC
$
maps $b$ and $b^{-1}$ to $1$ (recall that $\phi(b)^4=1$).
Therefore, the map 
\[
\res \circ\phi^*\co \mathcal{W} \to X(F_2,\SL(3,\CC))
\]
 maps
$\mathcal{W}$ onto the 2-dimensional component $V_2$ given by the equations
$\beta=\bar\beta=1$. On this component we have
  $\alpha=\zeta$, and  $\bar\alpha= ( \nu \bar{\nu}-2)-\zeta$ (i.e.  $\alpha$ and $\bar\alpha$ are 
the solutions of the equation on $\zeta$). In addition, the 
intersection
 with $\alpha=\bar \alpha$ corresponds to the two possible values of $\zeta$ (for 
 fixed $\nu$ and $\bar{\nu}$) being equal, this is to the zero set of the discriminant
 of the quadratic equation on $t$. 
 The parameters for $\chi\in \mathcal{W}$ are $\chi(k^{-1}l)$, $\chi(kl^{-1})$ and $\chi([k,l])$.
 Now, by \eqref{eq:phi_a-phi_b} we have 
 \[
 \phi(t)= k l k \sim k^2 l = k^{-1} l
 \quad \text{ and }\quad
 \phi(a)= k^{-1} l^{-1} k l \sim [k,l],
 \] 
 and hence the parameters for $\phi^*\chi\in X(\Gamma,\SL(3,\CC))$ 
 are
 \[ \nu = \phi^*\chi(t),\quad \bar{\nu} = \phi^*\chi(t^{-1})\quad\text{ and }\quad
\zeta= \phi^*\chi(a)\,.
 \]

We obtain the component $V_2$ by the same considerations and 
by replacing $\phi$ by $\phi\circ h$. Notice that by \eqref{eq:h} we have 
$h(t) = t^{-1}$ and $h(a)=b^{-1}$.
 \end{proof}

\begin{remark}
One may ask why the Dehn fillings $K(\pm 3)$ give new components of 
$X(\Gamma,\SL(3,\CC))$, while  $K(\pm 1)$ and $K(\pm 2)$ do not, even if 
all of them are small Seifert fibered orbifolds. 
The reason is that the groups of the base orbifolds are different
and their varieties of representations have different dimension: 
$X(\pi_1^{\mathcal{O}}(S^2(3,3,4)),\SL(3,\CC))$ has a component dimension two, though
$X(\pi_1^{\mathcal{O}}(S^2(2,q,r)),\SL(3,\CC))$ ($q,r\geq 2$) has dimension zero.
This can be checked with the same argument as in the proof of Lemma~\ref{lem:D334}.
\end{remark}

 \section{Parametrizing representations}
 \label{sec:explicite}

 Similar representations of knot groups into $\SL(3,\CC)$ have been used in the literature before.
 In particular Bing and Martin used them to prove Property P for twist knots
 (see \cite{Bing-Martin}). The study of representations of  $D(3,q,r)$ to $\SL(3,\CC)$ goes back to \cite{Coxeter}. Some of this is presented in \cite[Section~15B]{BZH2013}.
 
We consider the following two matrices $K$ and $L$ of $\SL(3,\CC)$:
\[
K =        	
\left(\begin{array}{rrr}
0 & 0 & 1 \\
x_{0} & 1 & x_{1} \\
-1 & 0 & -1
\end{array}\right)
\text{ and } \ \
L=
\left(\begin{array}{rrr}
1 & y_{0} & y_{1} \\
0 & -1 & -1 \\
0 & 1 & 0
\end{array}\right)\,.
\]
Notice that $K$ and $L$ are of order three, 
\[
KL=\left(\begin{array}{rrr}
0 & 1 & 0 \\
x_{0} & x_{0} y_{0} + x_{1} - 1 & x_{0} y_{1} - 1 \\
-1 & - y_{0} - 1 & - y_{1}
\end{array}\right)
\]
that the characteristic polynomial of $KL$ is given by
\[
P_{KL} (t)=        	
t^3 - \left( x_{0} y_{0} +  x_{1} - y_{1} - 1\right) t^2 +
\left(x_{0} y_{1} -  x_{1} y_{1} -  x_{0} -  y_{0} + y_{1} - 1\right) t- 1\,.
\]
Hence, $\tr(KL) = \tr(L^{-1}K^{-1}) =1$ if and only if 
\[
x_{0} y_{0} +  x_{1} - y_{1} - 2 =0 
\ \text{ and } \
x_{0} y_{1} -  x_{1} y_{1} -  x_{0} -  y_{0} + y_{1} - 2\,.
\]
Now, define the ideal 
\[ I = \left( x_{0} y_{0} +  x_{1} - y_{1} - 2, x_{0} y_{1} -  x_{1} y_{1} - 
x_{0} - y_{0} +  y_{1} - 2\right)\subset \CC[x_0,x_1,y_0,y_1]
% I = Ideal (x0*y0 + x1 - y1 - 2, x0*y1 - x1*y1 - x0 - y0 + y1 - 2) 
\]
and $X=\mathbf{V}(I)\subset \CC^4$ its zero set. 
For each point $(x_0,x_1,y_0,y_1)\in X$ we obtain a representation 
$\rho_{(x_0,x_1,y_0,y_1)}\colon D(3,3,4)\to\SL(3,\CC)$ mapping $\phi(b)=xk$ to a matrix 
$B$ such that $\tr(B)=\tr(B^{-1}) =1$.

An easy computation gives that $X$ is a smooth irreducible variety, that we view as a subvariety of 
$R(D(3,3,4),\SL(3, \CC))$. The following proposition says that we can view it as a birrational slice.

\begin{proposition}
 \label{prop:slice}
 The projection  $ R(D(3,3,4),\SL(3, \CC))\to X(D(3,3,4),\SL(3, \CC)) $ 
 restricts to a birrational map $X\to \mathcal{W}$. 
\end{proposition}

\begin{proof}
We write the projection restricted to $X$ as a regular map
$f\colon  X\to \mathcal{W} \subset \CC^3$ given by
\[
f(x_0,x_1,y_0,y_1) = \big(\nu(x_0,x_1,y_0,y_1), 
\bar\nu(x_0,x_1,y_0,y_1), \zeta(x_0,x_1,y_0,y_1)\big)
\]
where we have used the parameters of Lemma~\ref{lem:D334}
\[
\nu=\tr\rho_{(x_0,x_1,y_0,y_1)}(k^{-1}l),\ 
\bar\nu=\tr\rho_{(x_0,x_1,y_0,y_1)}(k l^{-1})\ \text{ and }\ 
\zeta =\tr\rho_{(x_0,x_1,y_0,y_1)}([k,l])\,.
\]
In the ambient coordinates the map $f$ is given by
\begin{align*}
\nu(x_0,x_1,y_0,y_1) &= x_{0} y_{0} -  x_{1} y_{0} + x_{0} + y_{1} - 2,
\\
\bar \nu(x_0,x_1,y_0,y_1)  &= x_{0} y_{1} -  x_{1} + y_{0} -  y_{1} + 1,
\\
\zeta(x_0,x_1,y_0,y_1)&=  x_{0}^{2} y_{0} y_{1} -  x_{0} x_{1} y_{0} y_{1} -  x_{0}^{2} y_{0} +
x_{0} y_{0}^{2} -  x_{1} y_{0}^{2}  \\
& \quad
+ x_{0} x_{1} y_{1} -  x_{1}^{2}
y_{1} -  x_{0} y_{0} y_{1} + x_{1} y_{0} y_{1} + x_{0} y_{1}^{2}  \\
& \qquad
- 
x_{0} x_{1} + 2 x_{0} y_{0} -  x_{1} y_{0} - 3 x_{0} y_{1} + 2 x_{1}
y_{1} + y_{0} y_{1} -  y_{1}^{2}  \\
& \qquad\quad
+ 4 x_{0} -  x_{1} - 2 y_{0} - 2\,.
\end{align*}

%  Defn of f: Defined on coordinates by sending (x0, x1, y0, y1) to
%  f(x0, x1, y0, y1) = ( x0*y0 - x1*y0 + x0 + y1 - 2, x0*y1 - x1 + y0 - y1 + 1,
%    x0^2*y0*y1 - x0*x1*y0*y1 - x0^2*y0 + x0*y0^2 - x1*y0^2 + x0*x1*y1 -
%   x1^2*y1 - x0*y0*y1 + x1*y0*y1 + x0*y1^2 - x0*x1 + 2*x0*y0 - x1*y0 -
%    3*x0*y1 + 2*x1*y1 + y0*y1 - y1^2 + 4*x0 - x1 - 2*y0 - 2)

The rational inverse  is the map $g\colon \mathcal{W}\to X$ given by
\begin{multline*}
g(\nu,\bar\nu,\zeta) = 
\Big(\frac{\nu^{2} + \nu\bar\nu - 2 \bar\nu -  \zeta - 3}{\zeta - 1},
 \frac{{\nu}^{2} -  {\bar{\nu}}^{2} + 2 {\nu} - 2 {\bar{\nu}} + {\zeta} - 1}{{\zeta} - 1},\\
 \frac{{\nu} {\bar{\nu}} -  {\bar{\nu}}^{2} + 2 {\nu} - 2 {\zeta} - 2}{- {\nu} {\bar{\nu}} + {\zeta} + 3},
 \frac{- {\nu}^{2} + 2 {\bar{\nu}} -  {\zeta} + 1}{- {\nu} {\bar{\nu}} + {\zeta} + 3}\Big)\,.
\end{multline*}
%Defn of g: Defined on coordinates by sending (nu,nub,zeta) to
%    g(nu,nub,zeta)=    ((nu*nub + nub^2 - 2*nu - zeta - 3)/(zeta - 1), 
% 		(-nu^2 + nub^2 - 2*nu + 2*nub + zeta - 1)/(zeta - 1), 
%			(-nu^2 + nu*nub + 2*nub - 2*zeta - 2)/(-nu*nub + zeta + 3), 
%				(-nub^2 + 2*nu - zeta + 1)/(-nu*nub + zeta + 3))
The map $g$ is defined off the algebraic set 
$Y=(X\cap\{ {\zeta}=1\})\cup (X\cap\{ {\nu} {\bar{\nu}} = {\zeta} + 3 \})$. 
The decomposition of $Y=Y_1\cup\cdots\cup Y_6$  into irreducible components
is obtained by computer supported calculations \cite{fig8html}, and is given by
\begin{alignat*}{2}
Y_1 &= \mathbf{V}(  {\zeta} - 1,\ {\nu} + {\bar{\nu}} + 2 ), &
Y_2 &= \mathbf{V}(  {\zeta} - 1,\  {\nu} + \eta^2\,{\bar{\nu}} -2\,\eta ),\\
Y_3 &= \mathbf{V}(  {\zeta} - 1,  {\nu} -\eta\,{\bar{\nu}} + 2\,\eta^2  ),&
Y_4 &= \mathbf{V}(   {\nu} + {\bar{\nu}} + 2,  {\bar{\nu}}^2 + 2\,{\bar{\nu}} + {\zeta} + 3 ),\\
Y_5 &= \mathbf{V}(   {\nu} + \eta^2\,{\bar{\nu}} -2\,\eta,  {\bar{\nu}}^2 + 2\,\eta^2\,{\bar{\nu}} -\eta\,{\zeta} -3\,\eta ),&\quad
Y_6 &= \mathbf{V}(   {\nu} -\eta\,{\bar{\nu}} + 2\,\eta^2,  {\bar{\nu}}^2 -2\,\eta\,{\bar{\nu}} + \eta^2\,{\zeta} + 3\,\eta^2)\,.
\end{alignat*}
Each $Y_i$ is isomorphic to an affine line and $\eta$ is a primitive $6$th root of unity.
\end{proof}
%[
%Closed subscheme of Affine Space of dimension 3 over Cyclotomic Field of order 6 and degree 2 defined by:
%  t - 1,
%  y + z + 2,
%Closed subscheme of Affine Space of dimension 3 over Cyclotomic Field of order 6 and degree 2 defined by:
%  t - 1,
%  y + (zeta6 - 1)*z + (-2*zeta6),
%Closed subscheme of Affine Space of dimension 3 over Cyclotomic Field of order 6 and degree 2 defined by:
%  t - 1,
%  y + (-zeta6)*z + (2*zeta6 - 2),
%Closed subscheme of Affine Space of dimension 3 over Cyclotomic Field of order 6 and degree 2 defined by:
%  y + z + 2,
%  z^2 + 2*z + t + 3,
%Closed subscheme of Affine Space of dimension 3 over Cyclotomic Field of order 6 and degree 2 defined by:
%  y + (zeta6 - 1)*z + (-2*zeta6),
%  z^2 + (2*zeta6 - 2)*z + (-zeta6)*t + (-3*zeta6),
%Closed subscheme of Affine Space of dimension 3 over Cyclotomic Field of order 6 and degree 2 defined by:
%  y + (-zeta6)*z + (2*zeta6 - 2),
%  z^2 + (-2*zeta6)*z + (zeta6 - 1)*t + (3*zeta6 - 3)
%]

This permits to give explicitly a representation  $\Gamma\to\SL(3,\CC)$ which corresponds to a given point $\mathcal{W}\setminus Y$.
\begin{exemple} Let us compute all irreducible representations $\rho\colon\Gamma\to\SL(3,\CC)$
such that $\chi_\rho\in V_1$ and $\tr\rho(m)=\tr\rho(m^{-1})=3$ i.e.\ $\nu=\bar\nu=3$.
By Proposition~\ref{prop:non-dist-components} we obtain $\zeta^2-7\zeta+14=0$ and hence
$\zeta_\pm = 7/2\pm i \sqrt{7}/2$. Now,
\[ 
g(3,3,\zeta_\pm) = \left(\frac{3}{2} \pm i \frac{\sqrt{7}}{2}  + , 1, 
\frac{1}{2} \mp i \frac{\sqrt{7}}{2}, \frac{3}{2} \mp i \frac{\sqrt{7}}{2}\right)\,.
\]
Hence 
\[ \rho(S)=
\begin{pmatrix}
 \frac{\sqrt{7}}{2} i  + \frac{3}{2} & 1 & 1 \\
 \frac{\sqrt{7}}{2} i  + \frac{5}{2} & -\frac{\sqrt{7}}{2} i  + \frac{5}{2} & 1 \\
 -\frac{\sqrt{7}}{2} i  - \frac{5}{2} & \frac{\sqrt{7}}{2} i  - \frac{3}{2} & -1
\end{pmatrix}, \qquad
 \rho(T)=
\begin{pmatrix}
 \frac{\sqrt{7}}{2} i  + \frac{5}{2} & -\frac{\sqrt{7}}{2} i  + \frac{5}{2} & 1 \\
 1 & -\frac{\sqrt{7}}{2} i  + \frac{3}{2} & 1 \\
 -\frac{\sqrt{7}}{2} i  - \frac{3}{2} & \frac{\sqrt{7}}{2} i  - \frac{5}{2} & -1
 \end{pmatrix}\,.
\]
Moreover, we obtain
\[
\rho(\ell) =
\begin{pmatrix}
3 i \, \sqrt{7} - 2 & \frac{3\sqrt{7}}{2} i + \frac{9}{2} & \frac{3\sqrt{7}}{2} i + \frac{3}{2} \\
3 i \, \sqrt{7} + 15 & -3 i \, \sqrt{7} + 10 & 9 \\
-\frac{3\sqrt{7}}{2} i - \frac{15}{2} & \frac{3\sqrt{7}}{2} i - \frac{9}{2} & -5
\end{pmatrix}
\]
is also an unipotent matrix. These representations where previously studied by Deraux and 
Falbel in connection with spherical CR structures on the complement of the figure eight knot
\cite{D1,D2,DF}.
\end{exemple}

\bibliographystyle{plain}
%  This inserts the bib file def2.bib
%  \bibliography{biblio-fig8}
%   \end{document}

\end{document}